\documentclass[a4paper,11pt]{amsart}
\usepackage{amssymb}
\usepackage{cmap}
\usepackage{hyperxmp}
\usepackage[pdfdisplaydoctitle=true,
            colorlinks=true,
            urlcolor=blue,
            citecolor=blue,
            linkcolor=blue,
            pdfstartview=FitH,
            pdfpagemode= UseNone,
            bookmarksnumbered=true]{hyperref}
\usepackage{todonotes}
\usepackage{mathtools}
\mathtoolsset{showonlyrefs}

\theoremstyle{plain}
\newtheorem{thm}{Theorem}[section]
\newtheorem{cor}[thm]{Corollary}
\newtheorem{lem}[thm]{Lemma}
\newtheorem{prop}[thm]{Proposition}

\newtheorem{conj}[thm]{Conjecture}

\theoremstyle{definition}

\theoremstyle{definition}
\newtheorem{rem}[thm]{Remark}

\newtheorem*{ass*}{Assumption}

\numberwithin{equation}{section}

\newcommand{\RR}{\mathbb{R}} 

\newcommand{\Diff}{\mathrm{Diff}}       

\newcommand{\norm}[1]{\left\Vert#1\right\Vert}          

\DeclareMathOperator{\tr}{tr} %

\usepackage{soul,xcolor}
\setstcolor{red}

\def\al{\alpha}

\def\th{\theta}

\let\on=\operatorname

\def\Diff{\on{Diff}}

\allowdisplaybreaks

\hypersetup{
    pdftitle={A diffeomorphism-invariant metric on the space of vector-valued one-forms}, 
    pdfauthor={Bauer}, 
    pdfsubject={MSC 2010: 58D05, 35Q35}, 
    pdfkeywords={Sobolev metrics, surfaces, shape analysis}, 
    pdflang=EN 
    }
\begin{document}

\title[The space of full-ranked one-forms]{A diffeomorphism-invariant metric on the space of vector-valued one-forms}

\author{Martin Bauer}
\address{Department of Mathematics, Florida State University, USA}
\email{bauer@math.fsu.edu}

\author{Eric  Klassen}
\address{Department of Mathematics, Florida State University, USA}
\email{klassen@math.fsu.edu}

\author[S.C. Preston]{Stephen C. Preston}
\address{Department of Mathematics, Brooklyn College and the Graduate Center, City University of New York, NY 11106, USA}
\email{stephen.preston@brooklyn.cuny.edu}

\author{Zhe  Su}
\address{Department of Mathematics, Florida State University, USA}
\email{zsu@math.fsu.edu}


\date{\today}

\begin{abstract}
 In this article we introduce a diffeomorphism-invariant Riemannian metric on the space of vector valued one-forms.
 The particular choice of metric is motivated by potential future applications in the field of functional data and shape analysis and
 by connections to the Ebin metric on the space of all Riemannian metrics.
 In the present work we calculate the geodesic equations and obtain
 an explicit formula for the solutions to the corresponding initial value problem.
 Using this we  show that it is a geodesically and metrically incomplete space and study the existence of totally geodesic subspaces.
 Furthermore, we calculate the sectional curvature and observe that, depending on the dimension  of the base manifold and the target space, it  either has a semidefinite sign or admits both signs.
\end{abstract}

\maketitle
\section{Introduction}
Motivated by applications in the field of mathematical shape analysis we introduce a diffeomorphism-invariant Riemannian metric on the space of full-ranked
$\mathbb R^n$-valued one-forms $\Omega^1_+(M,\mathbb R^n)$, where $M$ is a smooth, orientable, compact manifold (possibly with boundary) of dimension $m$ with $m\leq n$.
The definition of our metric will not include any derivatives of the tangent vectors. For this reason we call the metric an $L^2$-type metric, which however differs, due to the appearance of the foot point $\alpha$,  from the standard $L^2$-metric. The main reason for introducing this particular dependence on the foot point is the invariance of the resulting metric under the action of the diffeomorphism group $\Diff(M)$, see Lemma~\ref{lem:invariances}.
\\

\subsection*{Contributions of the article}
In this article we will initiate a detailed  study of the induced geometry of the proposed Riemannian metric.
The point-wise nature of the metric will allow us to reduce many of
the investigations of the metric to  the study of a finite dimensional space of matrices. Using this we are able to obtain
explicit formulas for geodesics and curvature.
Our main results of the article are as follows:
\begin{enumerate}
\item The induced geodesic distance on the space of full ranked, vector valued one-forms $\Omega^1_+(M,\mathbb R^n)$ is non-degenerate; see Theorem~\ref{thm:distpointwise} where a
 lower bound for the geodesic distance is obtained.
 \item The geodesic equation on the space of full ranked, vector valued one-forms $\Omega^1_+(M,\mathbb R^n)$ has explicit solutions for any initial conditions as presented in
Theorem~\ref{thm.geodesicformula.a}.
\item Depending on the values of $m$ and $n$ the sectional curvature is either sign-semidefinite or admits both signs.
\item The metric is linked via a Riemannian submersion to the Ebin metric on the space of all Riemannian metrics.
\end{enumerate}

As a consequence of the explicit formula for geodesics we will obtain the  metric and geodesic incompleteness of the space $\Omega^1_+(M,\mathbb R^n)$.
For the finite-dimensional space of matrices we will characterize its metric completion, which consists of a quotient space of
matrices, where two matrices are identified if they have less than full rank. In future work, we plan to use this characterization to determine
the metric completion of the space of full ranked one-forms, using a similar strategy as in \cite{clarke2013completion}.  Finally, in Section~\ref{shape:analysis},
we will discuss potential applications in the field of shape analysis, that have been further developed in the application-oriented article~\cite{su2019shape}.

\subsection*{Background and motivation}
In the following we will further motivate the study of this metric from two different angles.\\

\noindent\emph{Connections to shape analysis.}
The field of functional data analysis is concerned with describing and comparing data, where each data point can be a function \cite{srivastava2016functional,younes2010shapes,dryden2016statistical,bauer2014overview}. In this context the  difficulties lie both in the infinite dimensionality as well as in the non-linearity of the involved spaces. Infinite dimensional Riemannian geometry has proven to provide the necessary tools to tackle some of the problems and applications in this field. A space that is of particular interest in this area of research is the space of (unparametrized) curves or surfaces, which appears e.g., in the study of human organs, trajectory detection, body motions, or in general computer graphics applications. In order to obtain a Riemannian framework on the space of unparametrized surfaces (curves resp.),
one needs to consider metrics on the space of parametrized surfaces (curves resp.) that are invariant with respect to the reparametrization group \cite{michor2007overview,klassen2004analysis}.

Given a parametrized surface (curve resp.) $f\colon M\to \mathbb R^n$, we can view $df$ as a full-ranked one-form.
Hence, one can construct invariant Riemannian metrics on the space of parametrized surfaces (curves resp.)
as the pullback of invariant Riemannian metrics on the space of full-ranked one-forms, which puts us directly in the setup of this article.
A similar strategy has proven extremely efficient for shape analysis of unparametrized curves and has yielded the so-called SRV-framework \cite{klassen2004analysis,bauer2014constructing}. For surfaces the situation is more intricate. A generalization of the SRV-framework has been proposed in \cite{laga2017numerical}.
This framework, called the square root normal field (SRNF), has proved successful in applications but has some mathematical limitations, see e.g., the discussions in \cite{su2019shape}. The representation proposed in the current article will allow us to obtain a better mathematical understanding of the properties of the induced metric on the space of surfaces. The main reason is the  simpler characterization of the image of the map $f\mapsto df$, as compared to the SRNF.  In fact we obtain the isometric immersion:
\begin{align}
\operatorname{Imm}(M,\mathbb R^n) \longrightarrow\Omega_{+,\operatorname{ex}}^1(M,\mathbb R^n)\subset \Omega_+^1(M,\mathbb R^n)\;,
\end{align}
where $\Omega_{+,\operatorname{ex}}^1(M,\mathbb R^n)$ denotes the subset of exact one-forms (assuming that the topology of $M$ is sufficiently simple).
The present article will focus mainly on the  geometry on the larger space of all full-ranked one-forms;  we plan to study the submanifold geometry of
the space of exact one-forms  in future work.
This strategy is similar to that of Ebin-Marsden~\cite{ebin1970groups}, who considered the $L^2$-geometry of $\Diff(M)$ where all the geometry may be done point-wise, then considered the submanifold of volume-preserving diffeomorphisms under the induced metric (where geodesics describe ideal fluid motion).

In Figures~\ref{CurvesExamples}, ~\ref{SurfaceExample1}, and~\ref{SurfaceExample2} one can see  examples of geodesics in the space of immersions, equipped with the pull-back of the generalized Ebin metric studied in this article. These examples have been calculated using the numerical framework for the Riemannian metric studied in this paper as developed in~\cite{su2019shape}\footnote{An open source implementation of the corresponding numerical framework can be found at~\url{https://github.com/zhesu1/elasticMetrics}.}, where the spherical parametrizations of the boundary surfaces have been obtained using the code of Laga et al.~\cite{kurtek2013landmark}.

\begin{figure*}
	\begin{center}
		\includegraphics[width=0.65\linewidth]{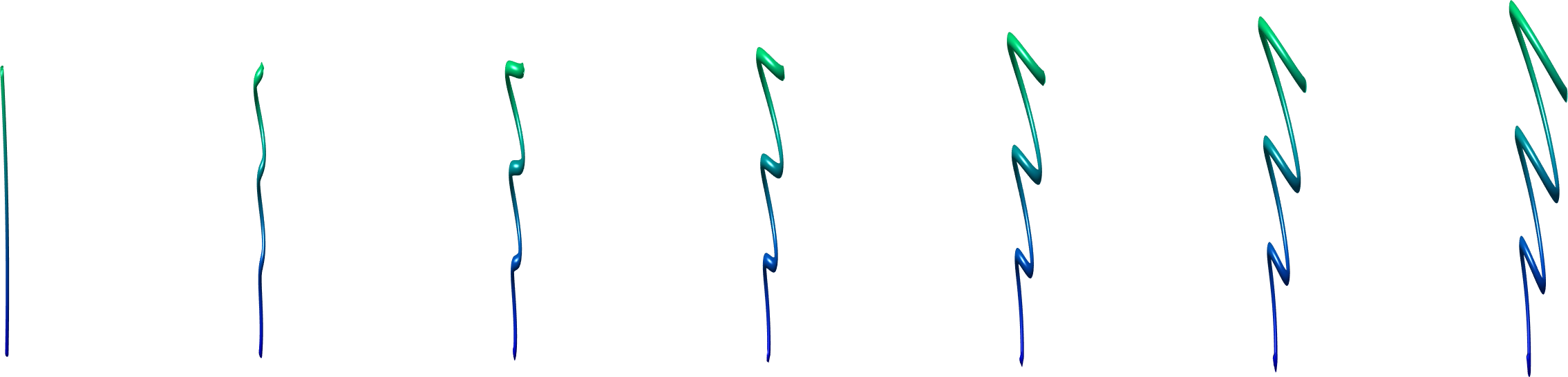}
	\end{center}
	\caption{A geodesic in the space of regular curves modulo translations with respect to the Younes-metric~\eqref{eq.metric.curve}, a special case of our metric.}
	\label{CurvesExamples}
\end{figure*}

\begin{figure*}
	\begin{center}
		\includegraphics[width=0.9\linewidth]{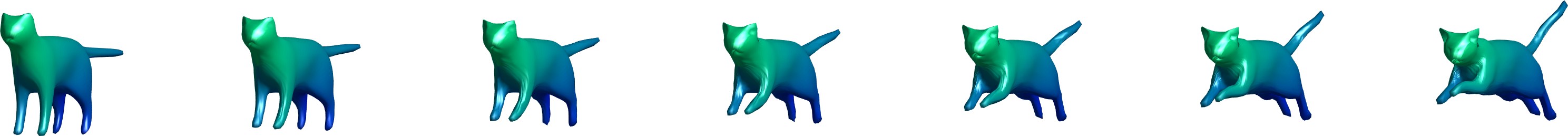}\\
		\vspace{.25cm}
		\includegraphics[width=0.89\linewidth]{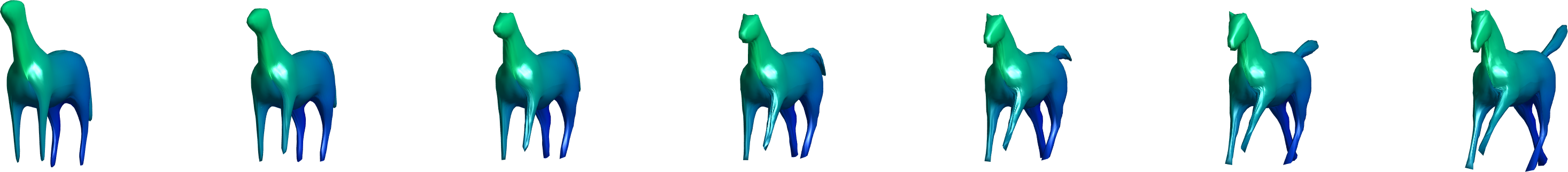}
	\end{center}
	\caption{Examples of geodesics in the space of surfaces modulo translations with respect to the generalized Ebin metric~\eqref{metric}. These examples have been calculated using the numerical framework for the Riemannian metric studied in this paper as developed in~\cite{su2019shape}.}
	\label{SurfaceExample1}
\end{figure*}

\noindent
\emph{Connections to the Ebin-metric on the space of all Riemannian metrics.}
Another motivation for the present article can be found in the connection of the proposed metric to the Ebin metric on the space of all Riemannian metrics, which has been introduced by Ebin~\cite{ebin1970manifold}; see also the  article of DeWitt \cite{DeWitt67}. Motivated by applications in Teichm\"uller theory, K\"ahler geometry and mathematical statistics, the geometry of this metric has been studied in detail by Clarke, Freed, Groisser, Michor, and others~\cite{gil1991riemannian,freed1989basic,clarke2010metric,clarke2013completion,clarke2013geodesics,clarke2011riemannian,bauer2013sobolev}. The proposed metric is closely related to the Ebin metric as they are connected via the Riemannian submersion:
\begin{equation}
\Omega_+^1(M,\mathbb R^n) \to \operatorname{Met}(M),\quad \alpha\mapsto \alpha^{T}\alpha;
\end{equation}
see Section~\ref{Sect:submersiontoEbin} for more details. Furthermore, the proposed metric shares many of the geometric features of the original Ebin metric, such as non-degenerate geodesic distance, existence of explicit solutions to the geodesic equation, and geodesic and metric incompleteness. On the other hand, we will see that the sectional curvature can admit both signs, which is in stark contrast to the Ebin metric on the space of Riemannian metrics, which always has non-positive curvature.
\subsection{Acknowledgements}
The authors want to thank the anonymous referees for their careful remarks that greatly improved the quality of the article. We are grateful to Hamid Laga for providing the parametrization of the boundary surfaces in Figure~\ref{SurfaceExample1}.

M. Bauer and Z. Su were partially supported by NSF-grant 1912037 (collaborative research in connection with NSF-grant 1912030).  E. Klassen was partially supported by Simons Foundation, Collaboration Grant for Mathematicians, no. 317865 and
S. C. Preston was partially supported by Simons Foundation, Collaboration Grant for Mathematicians, no. 318969.

\section{Notation}\label{notation}
\subsection{Spaces of matrices}\label{notation:matrices}
In large parts of the article the pointwise nature of the metric will allow us to reduce the analysis to the study of a corresponding Riemannian metric on a finite dimensional space of matrices. Therefore we introduce, for
$m\leq n \in \mathbb N$, the space of all full rank $n\times m$ matrices:
$$M_{+}(n,m):=\left\{a\in \mathbb R^{n\times m}|\;\operatorname{rank}(a)=m\right\}\;.$$

The space $M_{+}(n,m)$ is an open subset of the vector space of all $n\times m$-matrices $M(n,m)$ and is thus a manifold of dimension $n\times m$.
The full-rank condition on the elements of $M_{+}(n,m)$ allows us to consider the Moore-Penrose pseudo inverse $a^+$ of a matrix $a\in  M_{+}(n,m)$, which is defined by
$a^+=(a^Ta)^{-1}a^T$. The most important property of $a^+$ is $a^+a=I_{m\times m}$, i.e., $a^+$ is a left-inverse. Here $I_{m\times m}$ denotes the $m\times m$ identity matrix. In general we will use lower-case letters $u$ to denote $n\times m$ matrices in the tangent space, $u\in T_aM(n,m)\cong M(n,m)$, and we will use upper-case letters $U$ to denote the $n\times n$ square matrix $U = ua^+$. Note that the map $u\mapsto U = ua^+$, from $T_aM(n,m)$ to $M(n,n)$, is injective.

Related to the space of full rank $n\times m$ matrices is the space of positive definite symmetric $m\times m$-matrices:
\begin{equation}
\operatorname{Sym}_+(m):=\left\{a\in M(m,m): a^T=a \text{ and } a \text{ is positive definite} \right\}\;.
\end{equation}
Similarly to the space of all full rank $n\times m$ matrices, the space $\operatorname{Sym}_+(m)$ is a manifold as it is an open subset of a vector space, namely of the
space of all symmetric  $m\times m$-matrices $\operatorname{Sym}(m)$.

In the remainder of the article we will also use the group of all invertible $m$-dimensional matrices $\on{GL}(m)$, the
groups of special orthogonal matrices $\on{SO}(n)$ and $\on{SO}(m)$, and the groups of  orthogonal matrices $\on{O}(n)$ and $\on{O}(m)$.

\subsection{Spaces of  one forms, diffeomorphisms and Riemannian metrics}
Suppose $M$ is a compact $m$-dimensional  manifold $M$ and recall that $m \leq n$. Let $\Omega^1(M,\mathbb R^n)$ denote the space of smooth $\mathbb R^n$-valued one-forms on $M$. Recall that an $\mathbb R^n$-valued one-form $\alpha$ on $M$ is a choice, for each $x\in M$, of a linear transformation $\alpha(x):T_xM\to \mathbb R^n$ that varies smoothly with $x\in M$. Note that $\Omega^1(M,\mathbb R^n)$ is -- with the usual addition and scalar multiplication on $\mathbb R^n$ --  an infinite dimensional vector space.
If $\alpha(x)$ is injective for all $x\in M$, we say that $\alpha$ is a {\it full-ranked} one-form and we denote by $\Omega_+^1(M,\mathbb R^n)$ the space of full-ranked one-forms. We immediately obtain the following result concerning the manifold structure of $\Omega^1_+(M,\mathbb R^n)$ (see e.g., \cite{hamilton1982inverse} for an introduction to Fr\'echet manifolds):
\begin{lem}
The space of all full-ranked one-forms $\Omega^1_+(M,\mathbb R^n)$ is a smooth Fr\'echet manifold with tangent space the space of all one-forms
$\Omega^1(M,\mathbb R^n)$.
\end{lem}
\begin{proof}
By definition we have $\Omega^1_+(M,\mathbb R^n)\subset \Omega^1(M,\mathbb R^n)$. The full-rank condition is an open condition and thus $\Omega^1_+(M,\mathbb R^n)$ is an open subset of an infinite dimensional Fr\'echet space, which implies
the result.
\end{proof}

Related to this space is the infinite dimensional manifold of all smooth Riemannian metrics $\operatorname{Met}(M)$. For an overview on different Riemannian structures on this space and in particular to the Ebin metric, we refer to the vast literature; see e.g., \cite{gil1991riemannian,freed1989basic,clarke2010metric,clarke2013completion,clarke2013geodesics,clarke2011riemannian,bauer2013sobolev}. 

On both of the spaces we consider the action of the diffeomorphism group
\begin{equation}
\operatorname{Diff}(M):=\left\{\varphi \in C^{\infty}(M,M)|\; \varphi \text{ is bijective and }\varphi^{-1}\in C^{\infty}(M,M) \right\}
\end{equation}
via pullback:
\begin{align*}
&\Omega^1_+(M,\mathbb R^n)\times \operatorname{Diff}(M)\mapsto  \Omega^1(M,\mathbb R^n),\quad (\alpha,\varphi)\rightarrow \varphi^*\alpha(x)=\alpha(\varphi(x))\circ d\varphi(x)\,,\\
&\operatorname{Met}(M)\times \operatorname{Diff}(M)\mapsto  \operatorname{Met}(M),\quad (g,\varphi)\rightarrow \varphi^* g(x)=d\varphi^T(x) g(\varphi(x))d\varphi(x)\,.
\end{align*}
\section{A Riemannian metric on the space of full rank $n\times m$-matrices}
The main results of this article will be concerned with a diffeomorphism-invariant Riemannian metric on an infinite dimensional manifold of mappings, as introduced in the introduction \eqref{metric}.
The pointwise nature of the metric will allow us to reduce many aspects of the study of the corresponding geometry to the study of a corresponding metric on a (finite dimensional) manifold of matrices, which will be the object of interest in the following section.
Therefore we consider the space of full rank $n\times m$ matrices $M_+(n,m)$ with $m \leq n$ as introduced in Section~\ref{notation:matrices}.
For $a\in M_+(n, m)$ and $u, v\in T_aM_+(n, m)$ we define the Riemannian metric:
\begin{align}\label{eq.metric.a}
	\langle u, v \rangle_a = \tr(u(a^Ta)^{-1}v^T)\sqrt{\det(a^Ta)}.
\end{align}
Using the Moore-Penrose inverse $a^+ = (a^Ta)^{-1}a^T$ of $a\in M_+(n, m)$, 
we obtain an alternative formula for the metric that will turn out to be useful later:
\begin{align}\label{eq.metric.a.simplified}
	\langle u, v \rangle_a = \tr(UV^T)\sqrt{\det(a^Ta)}, \qquad U = ua^+, \quad V = va^+.
\end{align}
As a first result we will describe a series of invariance properties of the
Riemannian metric that will be of importance in the remainder of the article:
\begin{lem}\label{matmetinvar}
Let $a\in M_+(n, m)$ and $u, v\in T_aM_+(n, m)$.
	\begin{enumerate}
		\item The metric \eqref{eq.metric.a} is invariant under the left action of the orthogonal group:
		$$\langle zu, zv\rangle_{za} = \langle u, v\rangle_a \text{ for } z\in \operatorname{O}(n);$$
		\item The metric \eqref{eq.metric.a} satisfies the following transformation rule under the right action of the group of invertible matrices:
		$$\langle uc, vc\rangle_{ac} = \langle u, v\rangle_a|\det(c)| \text{ for } c\in\operatorname{GL}(m);$$
		\item The metric \eqref{eq.metric.a} is invariant under the right action of the group of determinant one or minus one matrices:
		$$\langle uc, vc\rangle_{ac} = \langle u, v\rangle_a \text{ for } c\in\operatorname{GL}(m), \det(c)= \pm1;$$
	\end{enumerate}
\end{lem}
\begin{proof}
	The proof consists of elementary matrix operations. For $z\in \operatorname{O}(n)$ we have
	\begin{align}
		\langle zu, zv\rangle_{za} &= \tr(zu(a^Tz^Tza)^{-1}v^Tz^T)\sqrt{\det(a^Tz^Tza)}\\
		&= \tr(u(a^Ta)^{-1}v^T)\sqrt{\det(a^Ta)} = \langle u, v\rangle_a,
	\end{align}
	which proves the invariance under the action of $\operatorname{O}(n)$.
	To see the second property we calculate for
	$c\in\operatorname{GL}(m)$:
	\begin{align}
		\langle uc, vc\rangle_{ac} &= \tr(uc(c^Ta^Tac)^{-1}c^Tv)\sqrt{\det(c^Ta^Tac)}\\
		&= \tr(ucc^{-1}(a^Ta)^{-1}(c^T)^{-1}c^Tv)\sqrt{\det(a^Ta)}|\det(c)|\\
		&= \langle u, v\rangle_a|\det(c)|.
	\end{align}
	The third statement follows immediately from the second one, which concludes the proof.
\end{proof}

\subsection{The space of symmetric $m\times m$-matrices}\label{sec:symmetric}
In this section we will describe the relation of our metric to a well-studied Riemannian metric on the space of symmetric matrices.
Therefore we  recall the definition of the finite dimensional version of the Ebin-metric, as studied by \cite{freed1989basic,clarke2010metric}:
\begin{align}\label{eq.metric.sym}
\langle h, k \rangle^{\operatorname{Sym}}_g = \frac14\tr(hg^{-1}kg^{-1})\sqrt{\det(g)},
\end{align}
where $g\in \operatorname{Sym}_+(m)$ and $h,k\in T_g \operatorname{Sym}_+(m)=\operatorname{Sym}(m)$.
Our main result in this section will show that the projection
\begin{equation}\label{projection}
\pi: M_+(n,m) \to \operatorname{Sym}_+(m), \qquad a \mapsto a^T a
\end{equation}
is a Riemannian submersion, where the spaces are equipped with their respective Riemannian metrics.


Note that $\operatorname{O}(n)$ acts by left multiplication on $M_+(n,m)$. The following proposition tells us that the orbits under this action are precisely the fibers of the map $\pi:M_+(n,m) \to \operatorname{Sym}_+(m)$ defined earlier.

\begin{prop}\label{prop.quotientidentification}
	Let $a, b\in M_+(n, m)$. Then $a^Ta = b^Tb$ if and only if there is $z\in \operatorname{O}(n)$ such that $a = zb$.
\end{prop}
\begin{proof}
	It is easy to see that if $a = zb$ for some $z\in \operatorname{O}(n)$, then
	\begin{equation}
	a^Ta = (zb)^Tzb = b^Tz^Tzb = b^Tb.
	\end{equation}
	Conversely, denote by $p\in\operatorname{Sym}_+(m)$ the positive definite symmetric square root of $a^Ta$. Then we have
	$$a^Ta = b^Tb = p^2 = \begin{pmatrix}
	p & 0
	\end{pmatrix}\begin{pmatrix}
	p\\0
	\end{pmatrix}, \text{ where }
	\tilde p= \begin{pmatrix}
	p\\
	0
	\end{pmatrix}\in M_+(n,m).$$ It is enough to show that there is $z\in \operatorname{O}(n)$ such that $a = z\tilde p$. Let $z_1 = ap^{-1}$. We have
	\begin{equation}
	z_1^Tz_1 = p^{-1}a^Tap^{-1} = I_{m\times m},
	\end{equation}
	which means that the columns in $z_1$ form a set of orthonormal vectors in $\mathbb R^n$. Let $z_2$ be an $n\times (n-m)$ matrix whose columns form an orthonormal basis of the orthogonal complement of the span of the columns of $z_1$. Let $z = \begin{pmatrix}
	z_1 & z_2
	\end{pmatrix}$. Then $a = z_1 p = \begin{pmatrix}
	z_1 & z_2
	\end{pmatrix}\tilde p = z\tilde p$. Now the conclusion follows by using
	\begin{equation}
	z^Tz = \begin{pmatrix}
	z_1^T\\z_2^T
	\end{pmatrix}\begin{pmatrix}
	z_1 & z_2
	\end{pmatrix} = \begin{pmatrix}
	z_1^Tz_1 & z_1^Tz_2\\
	z_2^Tz_1 & z_2^Tz_2
	\end{pmatrix} = I_{n\times n}.
	\end{equation}
\end{proof}
Proposition~\ref{prop.quotientidentification} implies that
$\pi$ induces a diffeomorphism
\begin{equation}\label{eq.identification}
\operatorname{O}(n)\backslash M_+(n,m)\cong\operatorname{Sym}_+(m),
\end{equation}
where $\operatorname{O}(n)\backslash M_+(n,m)$ denotes the space of orbits under the $O(n)$ action.
Furthermore, for any $a \in M_+(n,m)$ we obtain a (non-unique) decomposition
\begin{equation}\label{decomposition}
a= z  \begin{pmatrix}
 	s\\
	0_{(n-m)\times m}
 	\end{pmatrix},\qquad \text{with } z\in \operatorname{O}(n), \text{ and } s\in \operatorname{Sym}_+(m)\;.
\end{equation}
In the following theorem we describe the corresponding Riemannian submersion picture:

\begin{thm}\label{thm:matrices:submersion}
The mapping $\pi:  M_+(n,m) \to \operatorname{Sym}_+(m)$ is a Riemannian submersion, where $M_+(n,m)$ is equipped with the metric~\eqref{eq.metric.a}
and where $\operatorname{Sym}_+(m)$ carries the metric~\eqref{eq.metric.sym}. The corresponding vertical and horizontal bundles are given by:
\begin{align}
 	\mathcal V_a &= \{u\in M_+(n,m)\,|\, u = Xa, X \in \mathfrak{so}(n)\}\\
 	\mathcal H_a &= \{v\in M_+(n,m)\,|\, va^+\in \operatorname{Sym}(n)\}.
\end{align}
\end{thm}
\begin{proof}
In the following we identify the space of all symmetric matrices $\operatorname{Sym}_+(m)$ with the quotient space $\operatorname{O}(n)\backslash M_+(n,m)$. The Riemannian metric on $M_+(n,m)$ descends to a Riemannian metric on the quotient space due to the invariance under the left action of $\operatorname{O}(n)$. 
To determine the induced metric on the quotient space we need to calculate the vertical and horizontal bundle.
 	
It is immediate that the vertical bundle of $\pi$ at $a\in M_+(n,m)$ consists of all matrices $u$ such that $ u = Xa$ with $X\in\mathfrak{so}(n)$. A matrix $v$ is in the horizontal bundle if and only if it is orthogonal to all elements in the vertical bundle.

Letting $V=va^+$, we obtain
 	\begin{align}
 		0 = \langle Xa, v\rangle_a &= \tr(Xa(a^Ta)^{-1}v^T)\sqrt{\det(a^Ta)}\\
 		&= \tr(XV^T)\sqrt{\det(a^Ta)}.
 	\end{align}
	for all $X\in\mathfrak{so}(n)$.
It follows that $V$ has to be a symmetric matrix, proving the expressions for the vertical and horizontal bundles given in the statement of the Theorem.

 	To show that the differential $d\pi_a$ induces an isometry $\mathcal H_a\to T_{\pi(a)}\operatorname{Sym}_+(m)$ we calculate
 	\begin{align}
 		d\pi_a(v) = a^Tv + v^Ta.
 	\end{align}
 	For a horizontal tangent vector $v$ we have
 	\begin{align}
 		&\langle d\pi_a(v), d\pi_a(v)\rangle_{\pi(a)}^{\operatorname{Sym}}\\
 		&\qquad=\frac14\tr((a^Ta)^{-1}(v^Ta + a^Tv)(a^Ta)^{-1}(v^Ta + a^Tv))\sqrt{\det(a^Ta)}\\
 		&\qquad= \frac12\tr((a^Ta)^{-1}v^Ta(a^Ta)^{-1}v^Ta)\sqrt{\det(a^Ta)}\\
 		&\qquad\qquad\quad + \frac12\tr((a^Ta)^{-1}v^Ta(a^Ta)^{-1}a^Tv)\sqrt{\det(a^Ta)}
 	\end{align}
 	Using the cyclic permutation property of the trace and the fact that $V = va^+$ is symmetric we obtain
 	\begin{align}
 		&\tr((a^Ta)^{-1}v^Ta(a^Ta)^{-1}v^Ta)\sqrt{\det(a^Ta)}\\
 		&\qquad= \tr(a(a^Ta)^{-1}v^Ta(a^Ta)^{-1}v^T)\sqrt{\det(a^Ta)}\\
 		&\qquad= \tr(V^TV^T)\sqrt{\det(a^Ta)}= \tr(VV^T)\sqrt{\det(a^Ta)} = \langle v,v\rangle_a.
 	\end{align}
A similar calculation for the second term shows the statement.
 \end{proof}

\subsection{The Geodesic Equation}
In this section we will present the geodesic equation of the Riemannian metric \ref{eq.metric.a} and derive an explicit solution.
\begin{thm}\label{thm.geodesic.a}
	The geodesic equation on $M_+(n, m)$ with respect to the metric \eqref{eq.metric.a} is given by
	\begin{equation}
	 \label{eq.geodesic.a}
	\begin{aligned}
	a_{tt} = &a_t(a^Ta)^{-1}a_t^Ta + a_t(a^Ta)^{-1}a^Ta_t
	-a(a^Ta)^{-1}a_t^Ta_t\\
	&\qquad\qquad\qquad+\frac12\tr\left(a_t(a^Ta)^{-1}a_t^T\right)a-\tr\left(a_t(a^Ta)^{-1}a^T\right)a_t.
		\end{aligned}\end{equation}
\end{thm}
\begin{proof}
	Let $a(t)$ be a smooth curve in $M_+(n, m)$ defined on the unit interval $I = [0,1]$ and $\delta a$ be a smooth variation of $a$ that vanishes at the endpoints $t = 0$ and $t = 1$. The energy of $a$  in $M_+(n, m)$ is
	given by
	\begin{align}
	E(a) &= \int_I\langle a_t, a_t\rangle_a dt\\
	&=\int_I\tr(a_t(a^Ta)^{-1}a_t^T)\sqrt{\det(a^Ta)}dt.
	\end{align}
	The directional derivative of the energy function $E$ at $a$ in the direction of $\delta a$ can be calculated as:
	\begin{align}
	\delta E(a)(\delta a)
	&=\delta\left(\int_I \tr\left(a_t(a^Ta)^{-1}a_t^T\right)\sqrt{\det(a^Ta)} dt\right)(\delta a)\\
	&=2\int_I \tr\left((\delta a)_t(a^Ta)^{-1}a_t^T\right)\sqrt{\det(a^Ta)} dt\\
	&\qquad-2\int_I \tr\left(a_t(a^Ta)^{-1}(\delta a)^Ta(a^Ta)^{-1}a_t^T\right)\sqrt{\det(a^Ta)} dt\\
	&\qquad+\int_I \tr\left(a_t(a^Ta)^{-1}a_t^T\right)\delta\left(\sqrt{\det(a^Ta)}\right) dt.
	\end{align}
	Note that for any smooth matrix function $B:\mathbb R\to \operatorname{GL}(m)$ we have
	$$\dfrac{d}{dt} \det B = \tr\left(B_tB^{-1}\right)\det B; \quad \dfrac{d}{dt}B^{-1} = -B^{-1}B_tB^{-1}.$$
	Using integration by parts and the above formulas we obtain
	\begin{align}\label{eq.directional_derivative}
	\delta E(a)(\delta a) = \int_I\langle\mathcal{T}(a),\delta a\rangle_{a}dt,
	\end{align}
	where
	\begin{equation}
	\label{eq.T}
	\begin{aligned}
	\mathcal{T} (a) &=-2\tr\left(a^Ta_t(a^Ta)^{-1}\right)a_t-2a_{tt}+2a_t(a^Ta)^{-1}(a^Ta)_t\\
	&\qquad\qquad\qquad-2a(a^Ta)^{-1}(a_t)^Ta_t +\tr\left(a_t(a^Ta)^{-1}a_t^T\right)a.
	\end{aligned}\end{equation}
	Now the result follows, since $a$ is a geodesic if and only if $\mathcal{T}(a) = 0$.
\end{proof}
Using the Moore-Penrose inverse $a^+=(a^T a)^{-1}a^T$  a simpler form of the geodesic equation can be obtained:
\begin{lem}\label{lem.geodesic.L}
	Let $L = a_ta^+$. Then $a$ is a geodesic if and only if $L$ satisfies the equation:
	\begin{align}\label{eq.geodesic.L}
	L_t + \tr(L)L + (L^TL-LL^T) - \frac{1}{2} \tr(L^TL) aa^+ = 0
	\end{align}
\end{lem}
\begin{proof}
	We have
	\begin{align}
	L_t &= (a_ta^+)_t = a_{tt}a^+ + a_t \big((a^Ta)^{-1}a^T\big)_t \\
	&= a_{tt} a^+ - a_t (a^Ta)^{-1} (a^T_ta+a^Ta_t) (a^Ta)^{-1}a^T + a_t (a^Ta)^{-1}a^T_t \\
	&= a_{tt}a^+ - a_t (a^Ta)^{-1}a^T_t aa^+ - L^2 + a_t (a^Ta)^{-1}a^T_t.
	\end{align}
	Now equation \eqref{eq.geodesic.L} is obtained by inserting the expression of $a_{tt}$ in \eqref{eq.geodesic.a}.
\end{proof}
This form of the geodesic equation allows us to obtain an analytic formula for the solution of the geodesic initial value problem, which constitutes the first of the main results of this article:
\begin{thm}\label{thm.geodesicformula.a}
	Let $\delta = \tr(L^TL)$ and $\tau = \tr(L)$. The solution of \eqref{eq.geodesic.a} with initial values $a(0)$ and $L(0) = a_t(0)a(0)^+$ is given by
	\begin{equation}\label{eq.geodesicFormula}
	a(t) = f(t)^{1/m}e^{-s(t)\omega_0}a(0)e^{s(t)P_0},
	\end{equation}
	where
	\begin{align}
	f(t) &= \frac{m\delta(0)}{4}t^2+\tau(0)t+1,\qquad
	s(t) = \int_0^t\frac{d\sigma}{f(\sigma)},\\
	\omega_0 &= L^T(0) - L(0),\qquad P_0 = (a(0)^Ta(0))^{-1}(a_t(0)^Ta(0)) - \frac{\tau(0)}{m}I_{m\times m},
	\end{align}
	and $I_{m\times m}$ is the $m\times m$ identity matrix .
\end{thm}
\begin{proof}
This result can be shown by a direct calculation, substituting our solution into the geodesic equation. We can easily compute for example that
$$ L(t) = \frac{1}{f(t)} e^{-s(t)\omega_0} \left( \frac{f'(t)}{m} a_0 - \omega_0 a_0  + a_0 P_0 \right) a_0^+ e^{s(t)\omega_0},$$
and from here verify the formula \eqref{eq.geodesic.L}.
A more instructive proof of this result, along the lines of Freed-Groisser~\cite{freed1989basic} is presented in the Appendix~\ref{appendix.A}.
\end{proof}
In Figure~\ref{geo:rectangles} one can see a visualization of a geodesic in the space $M_+(3,2)$, where we visualize the matrices via their action on the unit rectangle.
\begin{figure}
\begin{center}
	\includegraphics[width=0.40\linewidth]{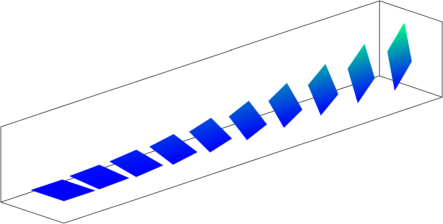}
	\includegraphics[width=0.40\linewidth]{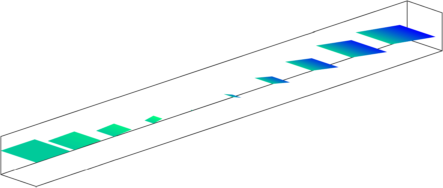}
	\end{center}
	\caption{Geodesics in the space $M_+(3,2)$. The matrices are visualized via their action on the unit rectangle. Note that the geodesic in the right figure leaves the space of full-ranked matrices in the middle of the geodesic.}
	\label{geo:rectangles}
\end{figure}
As a direct consequence we obtain the following result concerning the incompleteness of $M_+(n,m)$:
\begin{cor}\label{cor.aGeodesic}
	For any initial conditions $a(0) = a_0$ and $a_t(0)$ with $L_0 = a_t(0)a_0^+$, the geodesic $a(t)$ in $M_+(n,m)$ exists for all time $t\geq 0$ if and only if $a_t(0)$ is not a constant multiple $c$ of $a_0$ for some $c<0$. If $a_t(0)$ is a negative multiple of $a_0$, then the geodesic reaches the zero matrix at time $T = \frac{2}{|c|m}$.
\end{cor}
\begin{proof}
	Note that $L_0 = a_t(0)a_0^+ = a_t(0)a_0^+a_0a_0^+ = L_0a_0a_0^+$. Using the Cauchy-Schwarz inequality, we have
	\begin{align*}
	(\tr(L_0))^2 &=(\tr(L_0a_0a^+_0))^2 \leq\tr(L_0L_0^T)\tr(a_0a_0^+(a_0a_0^+)^T)\\
	&=\tr(L_0L_0^T)\tr(a_0a_0^+a_0a_0^+) =\tr(L_0L_0^T)\tr(a_0a_0^+)\\
	&=m\tr(L_0L_0^T).
	\end{align*}
	Then we conclude that $\tau_0^2\leq m\delta_0$ with $\tau_0 = \tau(0)$  and $\delta_0 = \delta(0)$ in the notation of Theorem~\ref{thm.geodesicformula.a}, and the only way the equality holds is if there is a number $c$ such that $L_0 =a_t(0)a_0^+= a_0a_0^+$, i.e., $a_t(0) = ca_0$. Thus if $a_t(0)$ is not a multiple of $a_0$, we must have $\tau_0^2< m\delta_0$, and therefore
	$$ f(t) = \epsilon^2 t^2 + (1+\tfrac{1}{2} \tau_0 t)^2, \qquad  s(t) = \frac{1}{\epsilon} \, \arctan{\left( \frac{2\epsilon t}{2+\tau_0 t}\right)}, \qquad \epsilon = \sqrt{m\delta_0 - \tau_0^2}.$$
	Thus $f(t)$ is never zero and $s(t)$ is well-defined for all $t>0$.
	
	On the other hand, if $a_t(0)= c a_0$, then $m\delta_0 =\tau_0^2$ and $\tau_0 = cm$, and we have
	$$ f(t) = (1 + \tfrac{cmt}{2})^2, \qquad s(t) = \frac{2t}{2+cm t}.$$
	Hence $f(t)$ approaches zero in finite time, and as it does, $s(t)$ approaches positive infinity. Note however that in this case $\omega_0=0$, and
	$$P_0 = c (a_0^T a_0)^{-1} (a_0^T a_0) - \frac{\tau_0}{m} I_m = c I_m- cI_m = 0.$$
	Thus the solution \eqref{eq.geodesicFormula} becomes
	$$ a(t) = (1+\tfrac{cmt}{2})^{2/m} a_0, $$
	and the result follows.	
\end{proof}

\subsection{Totally Geodesic Subspaces}
In this section we will study two families of totally geodesic subspaces of the space $M_+(n,m)$:
\begin{thm}\label{thm.totallyGeodesicSubspace.M}
	The following  spaces are totally geodesic subspaces of $M_+(n,m)$ with respect to the metric \eqref{eq.metric.a}:
	\begin{enumerate}
		\item the space $\operatorname{Scal}(b) :=\left \{ \lambda b|\lambda \in\mathbb R_{>0}\right\}$, where $b$ is any fixed element of  $M_+(n,m)$;
		\item the space $\operatorname{GL}(m)$, where elements in $\operatorname{GL}(m)$ are extended to $n\times m$ matrices by zeros.
	\end{enumerate}
\end{thm}
\begin{proof}
	The first result follows directly from the last sentence of the proof of Corollary~\ref{cor.aGeodesic}.

	To prove that each component of $\operatorname{GL}(m)$ is a totally geodesic submanifold, consider the map $M_+(n,m)\to M_+(n,m)$ defined by $a\mapsto Ja$, where
	$J$ is the matrix given in block diagonal form by
	\begin{equation}
	J=\begin{pmatrix} I_{m\times m}&0_{m\times (n-m)} \cr 0_{(n-m)\times m}&-I_{(n-m)\times (n-m)}\end{pmatrix}\;.
	\end{equation} We know that this map is an isometry by the first invariance proved in Lemma~\ref{matmetinvar}, since $J\in \operatorname{O}(n)$. Clearly, its fixed point set is $\operatorname{GL}(m)$. It is well known that each component of the fixed point set of any set of isometries is a totally geodesic submanifold -- see, for example \cite[Proposition 24]{petersen2006riemannian}. This proves that each component of $\operatorname{GL}(m)$ is a totally geodesic submanifold.

\end{proof}

\subsection{The Riemannian Curvature}
In this part we will calculate the Riemannian curvatures of the metric~\eqref{eq.metric.a}.  We will then show that the sectional curvature admits in general both signs.
There exists, however, an interesting subspace where the curvature is negative. In addition we will see that for the special case $m=1$,  all sectional curvatures are non-negative.

Since $M_+(n,m)$ is an open subset of the vector space of all matrices $M(n,m)$,
we have a global chart. Using this chart, we will always identify tangent vectors
of $M_+(n, m)$ with elements of $M(n,m)$.
To calculate the Riemannian curvature tensor, we will use the following curvature formula, which is true in local coordinates:
\begin{multline}\label{curvatureformulachris}
	R_a(u,v)w = -d\Gamma_a(u)(v,w) + d\Gamma_a(v)(u, w) \\+ \Gamma_a(u, \Gamma_a(v,w)) - \Gamma_a(v,\Gamma_a(u,w)),
\end{multline}
where $\Gamma: M_+(n, m)\times M(n, m)\times M(n, m)\to M(n, m)$ denotes the Christoffel symbols of the metric.
We can obtain the formula for the Christoffel symbol
by polarization of the right side of the geodesic equation $a_{tt} = \Gamma_{a}(a_t, a_t)$.
Using formula~\eqref{eq.geodesic.a} we thus get:
\begin{align}\label{eq.christoffel}
	\Gamma_{a}(u, v) &= \dfrac12\big(u(a^Ta)^{-1}v^Ta +  v(a^Ta)^{-1}u^Ta +  ua^+v +  va^+u - (ua^+)^Tv \\
	&\qquad - (va^+)^Tu +  \tr(u(a^Ta)^{-1}v^T)a -  \tr(ua^+)v - \tr(va^+)u\big).
\end{align}
From here it is a straightforward calculation to obtain the formula for the Riemannian curvature:
\begin{lem}\label{lem.RieCurvature}
	Using the notation $U = ua^+, V = va^+, W = wa^+$ the Riemannian curvature of $M_+(n, m)$ is given by
		\begin{equation}	\label{eq.RieCurvature}
				\begin{aligned}
	4(R_a&(u, v)w)a^+ \\&=[V, U^T]W^Taa^+ + W[U^T,V^T]aa^+ + WUV^Taa^++ W^TUV^Taa^+  \\
	&\qquad  + UWV^Taa^+ - [U, V^T]W^Taa^+ - WVU^Taa^+ - W^TVU^Taa^+ \\
	&\qquad - VWU^Taa^+ + 2VU^Taa^+W + WU^Taa^+V + VW^Taa^+U\\
	&\qquad - 2UV^Taa^+W  - WV^Taa^+U - UW^Taa^+V + 2aa^+VU^TW\\
	&\qquad + aa^+VW^TU + aa^+WU^TV - 2aa^+UV^TW - aa^+UW^TV\\
	&\qquad - aa^+WV^TU + [[V, U], W]+ [V^T, U^T]W + 2UW^TV\\
	&\qquad + 2UV^TW + V^TUW + W^TU^TV + V^TWU - 2VW^TU\\
	&\qquad - 2VU^TW - U^TVW - W^TV^TU - U^TWV\\
	&\qquad + \tr(VW^T)\tr(U)aa^+ -  \tr(V)\tr(WU^T)aa^+ + m\tr(UW^T)V \\
	&\qquad - m\tr(VW^T)U + \tr(W)\tr(V)U -\tr(W)\tr(U)V\
		    \end{aligned}
	\end{equation}
    Furthermore, if any of the tangent vectors
    of $u, v, w, s$ is of the form $\lambda a$ for $\lambda \in \mathbb R$, then
	\begin{equation}
		\langle R_a(u,v)w,s\rangle_a = 0.
	\end{equation}
\end{lem}
\begin{proof}
The proof is a very long, but basic computation using
the curvature formula \eqref{curvatureformulachris}
and the following formula for the differential of the Christoffel symbol;
\begin{align}
&2d\Gamma(u)(v, w)a^+ \\
=& -VU^TW^Taa^+ - VUW^Taa^+ + VW^TU - WU^TV^Taa^+\\
&\quad - WUV^Taa^+ + WV^TU - VU^Taa^+W - VUW + VU^TW \\
&\quad - WU^Taa^+V - WUV + WU^TV + aa^+UV^TW + U^TV^TW\\
&\quad - UV^TW +aa^+UW^TV + U^TW^TV - UW^TV -\tr(VU^TW^T)aa^+\\
&\quad - \tr(VUW^T)aa^+ + \tr(VW^T)U + \tr(VU^Taa^+)W + \tr(VU)W\\
&\quad - \tr(VU^T)W + \tr(WU^Taa^+)V + \tr(WU)V - \tr(WU^T)V.
\end{align}

\end{proof}

In the following we will decompose the tangent space of the space $M_+(n, m)$ in a scaling part -- i.e., changing only the determinant of the linear mapping -- and the complement.
Therefore we recall that any square matrix $U$ can be decomposed into a traceless part and a remainder as follows:
\begin{align}
	U = U - \frac{\tr(U)}{m}aa^+ + \frac{\tr(U)}{m}aa^+:=U_0+ \frac{\tr(U)}{m}aa^+.
\end{align}
Analogously we define for a non-square matrix $u\in T_aM_+(n,m)$ the decomposition
\begin{align}
	u = u - \frac{\tr(ua^+)}{m}a + \frac{\tr(ua^+)}{m}a:=u_0+\frac{\tr(ua^+)}{m}a.
\end{align}
Note that these two terms in the formula above are orthogonal with respect to the metric \eqref{eq.metric.a}. We will call $u_0$ the \emph{traceless part} and $\tfrac{\tr(ua^+)}{m}a$ the \emph{pure trace part} of $u$. It is easy to see that $U_0 = u_0a^+$.
We have seen in Lemma~\ref{lem.RieCurvature} that the curvature tensor vanishes if pure trace directions are involved. As a consequence, the
sectional curvature will only depend on the traceless part of the tangent vectors $u$ and $v$:
\begin{thm}\label{thm.SectionalCurvature}
	The sectional curvature of $M_+(n, m)$ at $a$ is given by
	\begin{align*}
	&4\mathcal K_a(u, v)/\sqrt{\det(a^Ta)} = 4\langle R(u,v)v,u\rangle_a/\sqrt{\det(a^Ta)}\\
	= &2\tr([V_0,U_0][V_0^T,U_0])+ 2\tr([V_0,U_0^T][V_0,U_0]) + 2\tr(V_0U_0V_0^TU_0^T)\\
	&\quad + \tr(V_0V_0^TU_0^TU_0) - 4\tr(V_0V_0U_0^TU_0^T) + 4\tr(V_0U_0^TU_0V_0^T) \\
	&\quad + \tr(V_0^TV_0U_0U_0^T) - 2\tr(V_0V_0^TU_0U_0^T) - 2\tr(V_0U_0^TV_0U_0^T)\\
	&\quad + 6\tr(V_0U_0^TV_0U_0^Taa^+) - 3\tr(V_0U_0^TU_0V_0^Taa^+) - 3\tr(U_0V_0^TV_0U_0^Taa^+) \\
	&\quad - m\tr(V_0V_0^T)\tr(U_0U_0^T) + m(\tr(U_0V_0^T))^2,
	\end{align*}
	where $u,v\in T_aM_+(n, m)$ are orthonormal with respect to  the metric \eqref{eq.metric.a}, and $U_0, V_0$ are the traceless parts of $U =ua^+$ and $V = va^+$, respectively.
	Furthermore, we have:
	\begin{enumerate}
	\item If one of the tangent vectors $u, v$ is a pure trace direction, then the sectional curvature is zero.
	 \item If $m\geq 2$ and $u, v\in T_aM_+(n, m)$ such that $U= ua^+$ and $V = va^+$ are symmetric -- i.e.,  for horizontal tangent vectors with respect to the projection~\eqref{projection} --
	then the sectional curvature is negative.
	\item If $m=1$, then all sectional curvatures are non-negative, and they vanish identically for $n=m+1=2$.\label{sec:curvature:m1}
	 \item If $ m\in \left\{2,3\right\}$ and $n\geq m+2$, then the sectional curvature always admits both signs.
	\end{enumerate} 	
	\end{thm}
\begin{rem}[Open cases and conjecture]
Using extensive testing with random matrices in MATLAB, we did not find any positive sectional curvatures for any of the open cases, i.e., for $m>3$, or for $m=\{2,3\}$ and $n=m+1$.
This leads us to the conjecture  that the sectional curvature is non-positive in these cases.
In Figure~\ref{randomexperiments} we show histogram plots of our random-matrix experiments, that also demonstrate the scarcity of positive sectional curvature in the case
 $m=\{2,3\}$ and  $n\geq m+2$.
\end{rem}	
\begin{figure}
\begin{center}
	\includegraphics[width=0.31\linewidth]{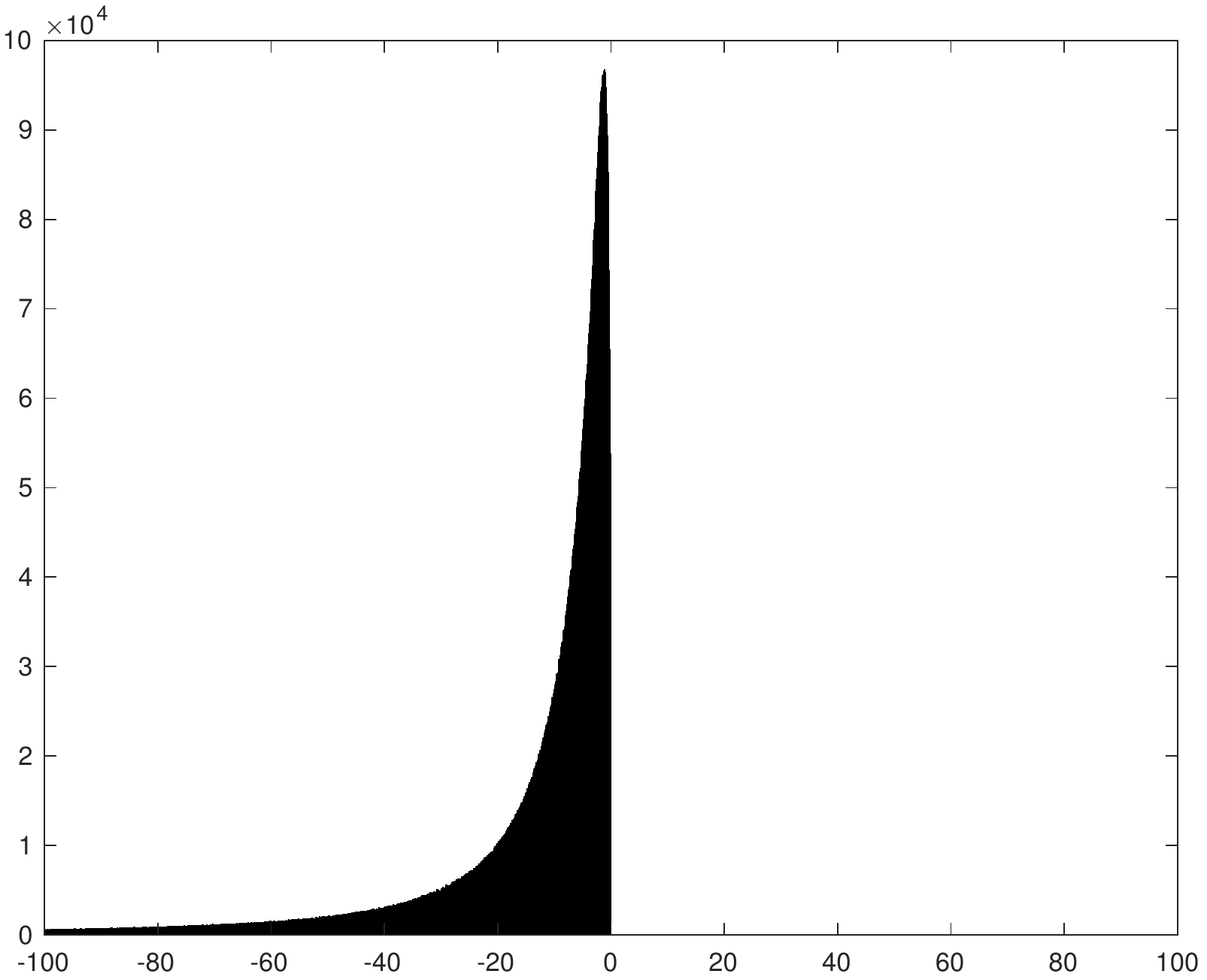}
	\includegraphics[width=0.31\linewidth]{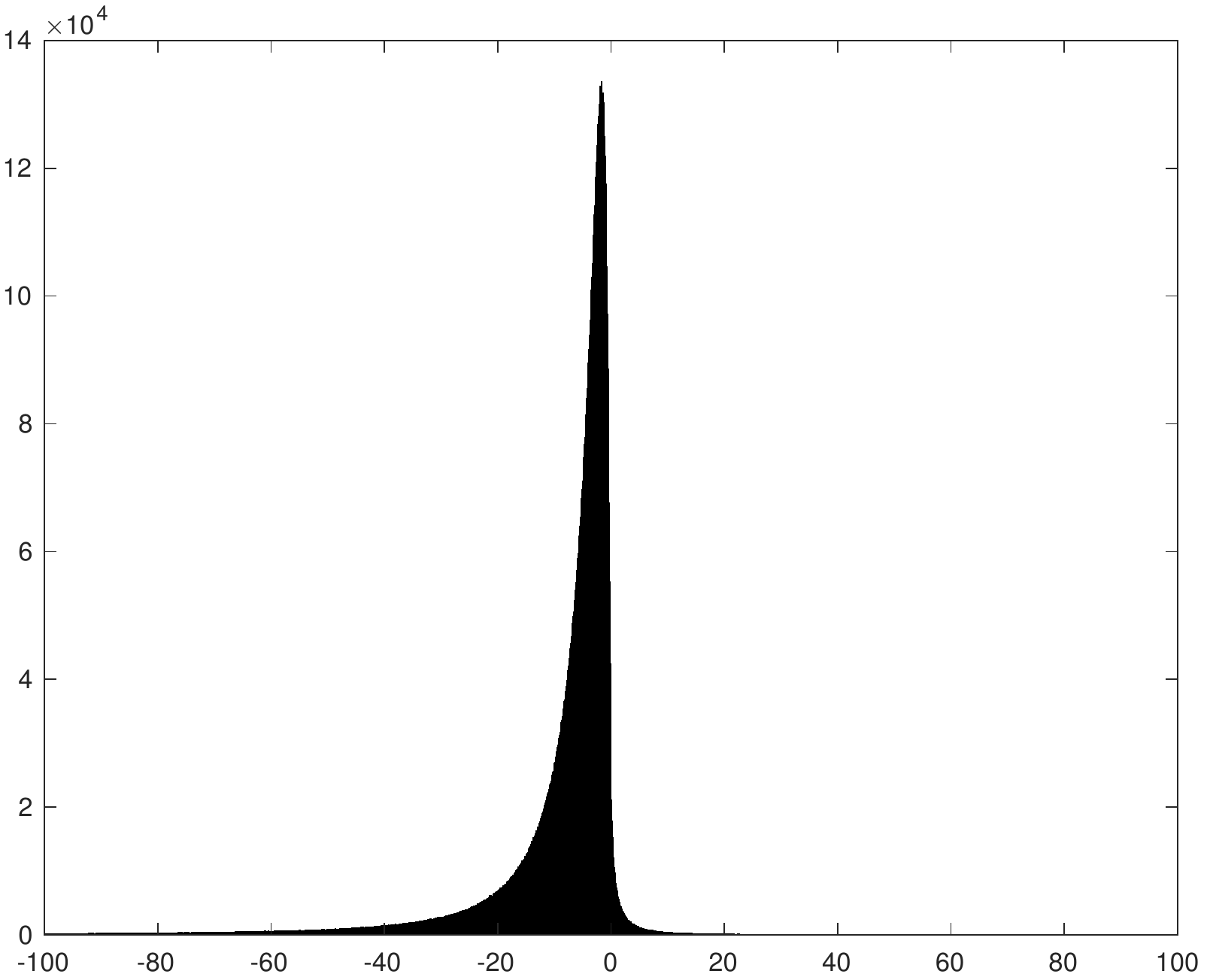}
	\includegraphics[width=0.31\linewidth]{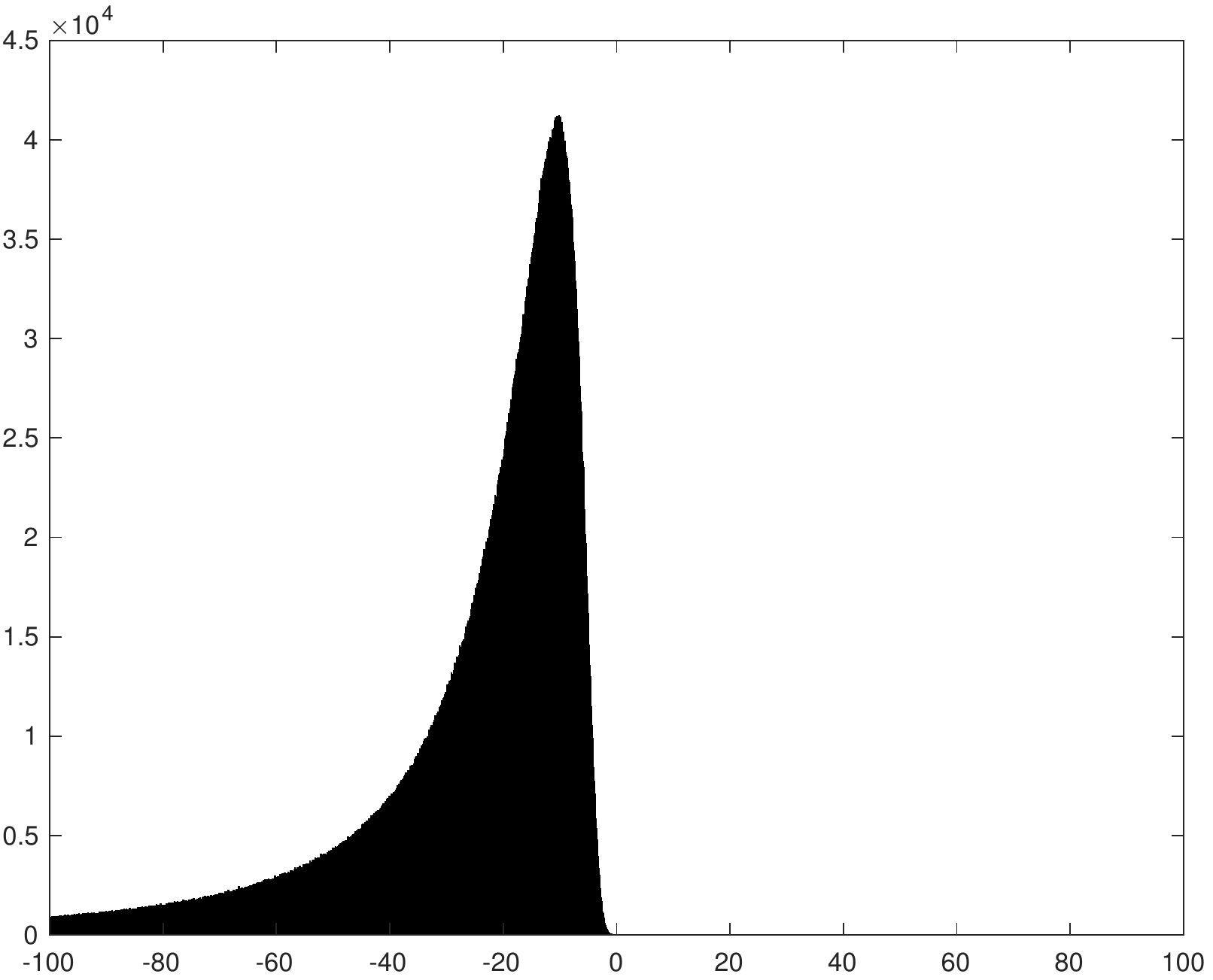}
	\end{center}
	\caption{Histogram plots demonstrating the scarcity of positive sectional curvature: $x$-axis: value of the sectional curvature; $y$-axis: number of 2-planes that attained this value. Left figure: $m=2$, $n=3$. Percentage of positive sectional curvature: zero.
	Middle figure:  $m=2$ , $n=4$. Percentage of positive sectional curvature: 3.041\%.  Right figure: $m=3$ , $n=5$. Percentage of positive sectional curvature: 0.007\%.
	The figures have been created in MATLAB using $10^7$ runs with random matrices for each choice of $m$ and $n$.}
	\label{randomexperiments}
\end{figure}
\begin{proof}[Proof of Theorem~\ref{thm.SectionalCurvature}]
	The formula for $\mathcal K$ at $a\in M_+(n,m)$ can be obtained by direct computation.
	W.l.o.g. we assume that $m$ and $n$ are not both one, as for this case the space is one-dimensional and the curvature is trivial.
For orthonormal $u$ and $v$ we have
	\begin{align}
	\mathcal K_a(u, v) = \langle R_a(u,v)v, u\rangle_a = \langle R_a(u_0,v_0)v_0, u_0\rangle_a,
	\end{align}
	 where the second equality is obtained by Lemma~\ref{lem.RieCurvature}.
	
	 Statement~(1) follows directly from the curvature formula. To see~(2) we calculate
	 \begin{align}
		&4\langle R(u,v)v, u\rangle_a/\sqrt{\det(a^Ta)}\\
		&\qquad= 14(\tr(U_0V_0U_0V_0) - \tr(U_0U_0V_0V_0))\\
		&\qquad\qquad + m\tr(U_0V_0)\tr(V_0U_0) - m\tr(U_0U_0)\tr(V_0V_0)\\
		&\qquad= 7\tr\left([U_0, V_0]^2\right) + m\left(\left(\tr(U_0V_0)\right)^2 - \tr(U_0^2)\tr(V_0^2)\right).
	\end{align}
	Note that $U, V$ being symmetric implies that $U_0, V_0$ are symmetric. Thus their commutator is antisymmetric and then $\tr\left([U_0, V_0]^2\right)\leq 0$. In addition, by the Cauchy-Schwarz inequality we have
	\begin{align}
		\left(\tr(U_0V_0)\right)^2 = \left(\tr(U_0V_0^T)\right)^2 \leq \tr(U_0U_0^T)\tr(V_0V_0^T) = \tr(U_0^2)\tr(V_0^2).
	\end{align}
	Therefore, $\mathcal K_a(u, v)\leq 0$. Note that we needed $m \geq 2$ to construct two linear independent tangent vectors $u$ and $v$ with $U$ and $V$ being symmetric.
Furthermore the inequality is strict if $U=U_0$ and $V=V_0$, i.e., if the linearly independent vectors $u,v$ are traceless. Note that such pairs always exist for $m \geq 2$.

For point~(3), we first observe that in the situation $m=1$, $u_0^Tv_0 = \norm{a}_a^2\langle u_0, v_0 \rangle_a$ with  $\norm{a}_a^2 = \sqrt{a^Ta}$, and
	$a^+ = (a^Ta)^{-1}a^T = \norm{a}^{-4}a^T$,
	where the norm $\norm{\cdot}$ is with respect to the metric \eqref{eq.metric.a}. By direct calculation we obtain
	\begin{align}
		&U_0U_0=V_0V_0=U_0V_0=V_0U_0=0_{n\times n},\quad  U_0^Ta = V_0^Ta=0_{n\times 1},
	\end{align}
	and
	\begin{align}
		U_0^TU_0 &= (a^+)^Tu_0^Tu_0a^+ = \norm{a}_a^{-6}\norm{u_0}_a^2aa^T,\\
		V_0^TV_0 &= (a^+)^Tv_0^Tv_0a^+ = \norm{a}_a^{-6}\norm{v_0}_a^2aa^T,\\
		U_0^TV_0 &= (a^+)^Tu_0^Tv_0a^+ = \norm{a}_a^{-6}\langle u_0, v_0\rangle_aaa^T.
	\end{align}
	Substituting these formulas into the formula of the sectional curvature for the general case and simplifying it, we have
	\begin{align}\label{eq.sectional.curvature.d=1}
		\mathcal{K}_a(u, v) = \frac34\norm{a}_a^{-2}\left(\norm{u_0}_a^2\norm{v_0}_a^2 - \langle u_0, v_0\rangle_a^2\right).
	\end{align}
	By the Cauchy-Schwarz inequality, the sectional curvature is therefore non-negative. If $n =2$ in addition, we have at each $a\in M_+(2,1)$ only one $2$-dim tangent plane. Let $u, v\in M_+(2,1)$ be a pair of orthonormal tangent vectors respect to the metric \eqref{eq.metric.a} such that $u$ is in the direction of $a$. Then $u_0=0$, and thus by formula~\ref{eq.sectional.curvature.d=1} the sectional curvature vanishes.

	Finally for statement (4), i.e., $m\in \{2,3\}$ and $n\geq m+2$,  we let
	\begin{align}a = \begin{pmatrix}
	\operatorname{Id}_{m\times m}\\
	0_{(n-m)\times m} \\
	\end{pmatrix},
	 u = \begin{pmatrix}
		0 & \cdots & 0 & 0\\
		\vdots & & \vdots & \vdots\\
		0 & \cdots & 0 & 0\\
		0 & \cdots & 0 & 1\\
		0 & \cdots & 0 & 0\\
	\end{pmatrix} , v = \begin{pmatrix}
	0 & \cdots & 0 & 0\\
	\vdots & & \vdots & \vdots\\
	0 & \cdots & 0 & 0\\
	0 & \cdots & 0 & 0\\
	0 & \cdots & 0 & 1\\
	\end{pmatrix},
	\end{align}
where $\operatorname{Id}_{m\times m}$ denotes the $m\times m$ identity matrix and $0_{(n-m)\times m}$ the $(n-m)\times m$ zero matrix.
It is easy to check that $u$ and $v$ are orthonormal tangent vectors at $a$ with respect to the metric \eqref{eq.metric.a}. Plugging $a$ and $u, v$ into the formula of the sectional curvature we obtain
	\begin{align}
		\mathcal K_a(u,v) = 4 - m,
	\end{align}
which proves the last statement. 	
\end{proof}

\subsection{The metric completion}
In Corollary~\ref{cor.aGeodesic} we have seen that $M_+(m,n)$ with the metric \eqref{eq.metric.a} is geodesically incomplete. By the theorem of Hopf-Rinow that implies that the corresponding metric space is also metrically incomplete.
In this section we will study its metric completion. For technical reasons we will restrict ourself to the case $n>m$, as  the space $M_+(m,m)=\operatorname{Gl}(m)$ is not connected and thus one would have to study the completion of each of the two connected components separately. To keep the presentation simple we will not treat this special case.

We first recall the formula for the geodesic distance function on $M_+(n,m)$ with respect to the metric \eqref{eq.metric.a}:
\begin{multline}\label{eq.distance.nbym}
	\on{dist}^{n\times m}(a_0, a_1) = \inf_{a}\Big\{L(a) = \int_I\|a_t(t)\|_{a(t)}dt\, \Big\vert \, a\colon[0,1]\to M_+(n,m)\\ \text{ is  piecewise differentiable with } a(0) = a_0, a(1) = a_1\Big\},
\end{multline}
where the norm $\|\cdot\|$ is induced by the metric \eqref{eq.metric.a} on $M_+(n,m)$.
We first calculate an upper bound for the geodesic distance:
\begin{lem}\label{lem.inequality.distance}
	Let $a, b\in M_+(n,m)$ with $n>m$. Then
	\begin{align*}
		\on{dist}^{n\times m}(a, b)\leq \dfrac{2}{\sqrt{m}}\left(\sqrt[4]{\det(a^Ta)}+ \sqrt[4]{\det(b^Tb)}\right).
	\end{align*}
\end{lem}
\begin{proof}
Let  $a,b  \in M_+(n,m)$. Using the invariance properties of the metric -- c.f. item (2) in Lemma \ref{matmetinvar} -- we observe that the geodesic distance between scaled versions of the matrices $a$ and $b$ can be made arbitrary small, i.e., given $\epsilon>0$ there exists $\delta>0$ such that $\on{dist}^{n\times m}(\delta a, \delta b)\leq \epsilon$.

We will now calculate an upper bound for the geodesic distance between a matrix  to a scaled version of the same matrix. Assume $\epsilon, \delta$ are as above and let $a_1\in M_+(n,m)$. We consider the path $a(t) =(1-t)a_1$ for $t\in (0,1-\delta)$.  Using  $a_t(t) = -a_1$ we calculate
	\begin{align}\label{eq.ine1}
		\on{dist}^{n\times m}(a_1,\delta a_1 ) \leq& \int_0^{1-\delta}\|a_t(t)\|_{a(t)}dt
		\leq \int_0^{1}\|a_t(t)\|_{a(t)}dt\notag\\
		 =&\int_0^1\left(\tr\left(a_t\left(a^T(t)a(t)\right)^{-1}a_t^T\right)\sqrt{\det(a^T(t)a(t))}\right)^{1/2}dt\notag\\
		=&\int_0^1\left(mt^{m-2}\sqrt{\det(a_1^Ta_1)}\right)^{1/2}dt = \dfrac{2}{\sqrt{m}}\sqrt[4]{\det(a_1^Ta_1)}.
	\end{align}
Now the statement follows from the triangle inequality:
	\begin{align*}
		\on{dist}^{n\times m}(a, b) \leq& \on{dist}^{n\times m}(a, \delta a) + \on{dist}^{n\times m}(\delta a, \delta b) +\on{dist}^{n\times m}(\delta b, b)\\
		=&\dfrac{2}{\sqrt{m}}\left(\sqrt[4]{\det(a^Ta)}+ \sqrt[4]{\det(b^Tb)}\right)+\epsilon\;,
	\end{align*}
	which proves the result.
\end{proof}
Using this result we are able to characterize the metric completion of $M_+(n,m)$:
\begin{thm}
Let $n>m$. The metric completion of the space $M_+(n,m)$ with respect to the geodesic distance \eqref{eq.distance.nbym} is given by
$M(n,m)/\sim$ where $a\sim b$ if $\on{rank}(a)<m$ and $\on{rank}(b) <m$.
\end{thm}

\begin{proof}

	In the following let $\{a_k\}$ and $\{b_k\}$ be Cauchy sequences with respect to the geodesic distance function
	$\on{dist}^{n\times m}$.
	First we consider the case that $\operatorname{det}(a_k^Ta_k) \to 0$ and $\operatorname{det}(b_k^Tb_k) \to 0$ as
	$k$ goes to infinity.
	By Lemma~\ref{lem.inequality.distance} we  have $\on{dist}^{n\times m}(a_k, b_k)\to 0$ and thus any two such sequences are identified with each other in the metric completion. This new point corresponds to the identification of all matrices with non-maximal rank.

	It remains to consider the case in which $\operatorname{det}(a_k^Ta_k)\not\to 0$ as $k\to\infty$. In this case, there exists a subsequence of $\tilde a_k$, an $\eta>0$ and $K_0\in \mathbb N$ such that $\operatorname{det}(\tilde a_k^T \tilde a_k)>\eta$ for all $k>K_0$.
	By the identification~\eqref{eq.identification} we write $\tilde a_k=z_ks_k$ with $z_k\in \on{O}(n)$ and $s_k\in \on{Sym}^+(m)$ (extended to a $n\times m$ matrix with zeros). We will view $s_k$ both as an $n \times m$ and as an $m\times m$ matrix, depending on which form is more convenient for our purposes. Since $\on{O}(n)$ is compact we can always pass to a convergent subsequence and using the left invariance of the Riemannian metric (and thus of the induced geodesic distance function) we may assume that this limit is the identity matrix, i.e., $\lim_{k\to \infty}z_k =I_{n\times n}$. It remains to show that $s_k$ converges. Let $\epsilon>0$. Since $\tilde a_k$ is a Cauchy sequence, for all $k, l$ sufficiently large we have
	\begin{align*}
	\epsilon &> \on{dist}^{n\times m}(z_ks_k,z_ls_l)=
	\on{dist}^{n\times m}(s_k,z_k^T z_l s_l)
	\geq \inf_{z\in \on{O}(n)} \on{dist}^{n\times m}(s_k,zs_l).
	\end{align*}
	The mapping $\pi: a\mapsto a^Ta$ is a Riemannian submersion onto the space of symmetric matrices with the metric \eqref{eq.metric.sym} and thus the last expression is equal to the geodesic distance induced by \eqref{eq.metric.sym} of $s_k^Ts_k$ and $s_l^Ts_l$.
	Thus we have shown that $s_k^Ts_k \in \on{Sym}^+(m)$
	is a Cauchy sequence with respect to the geodesic distance of the metric~\eqref{eq.metric.sym}. By a result of Clarke~\cite[Proposition 4.11]{clarke2013completion} and the assumption on the determinant, there exists a constant $C$ such that $(s_k^Ts_k)_{ij}\leq C$ for all $k >K_0$. It follows that $$(s_k^Ts_k)_{jj} = \sum_i(s_k)_i^j(s_k)_i^j\leq C$$ and thus $|(s_k)_i^j|\leq \sqrt{C}$. Therefore $s_k$ is in a bounded and closed subset of $\mathbb R^{m\times m}$ and thus, by taking a further subsequence, we can conclude that $s_k$ converges to a unique element $s \in \operatorname{Sym}^+(m)$.
\end{proof}

\begin{rem}[The space of symmetric matrices (revisited)]
Using the Riemannian submersion structure as described in Section~\ref{sec:symmetric} to study the geometry
of the space of symmetric matrices~\ref{eq.metric.sym}, one can regain several classical results of \cite{freed1989basic,gil1991riemannian,clarke2013completion,ebin1970manifold}, including the solution
for the geodesic equation and  the non-positivity of the sectional curvature. We will present the alternative derivation
of these results in Appendix~\ref{appendix:submersion}.
\end{rem}

\section{The generalized Ebin metric}
In this section we will introduce the generalized Ebin metric on the space of one-forms $\Omega_+^1(M,\RR^n)$.
Therefore let $\alpha\in\Omega_+^1(M,\RR^n)$. Then the tensor product $\alpha^T\otimes\alpha$, which is defined for each $x\in M$ as the pull back of the Euclidean scalar product under $\alpha$, defines a Riemannian metric on $M$. For simplicity, we will just denote this tensor product (Riemannian metric resp.) by $\alpha^T\alpha$. Consequently, the inner product $(\alpha^T\alpha)_x = (\alpha^T\alpha)(x)$ induces for each $x\in M$ an inner product on the cotangent space $T_x^*M$, which is given by $(\alpha^T\alpha)^{-1}_x = ((\alpha^T\alpha)(x))^{-1}$.

Given $\zeta,\eta\in T_{\alpha}\Omega_+^1(M,\RR^n)$, we can now define a Riemannian metric on $\Omega_+^1(M,\RR^n)$ as the integral over $M$ of the point-wise inner product of $\zeta_x = \zeta(x)$ and $\eta_x = \eta(x)$ with respect to the volume form $\on{vol}(\alpha)$ associated with the metric $\alpha^T\alpha$ on $M$:
	\begin{align}\label{metric}
	G_{\alpha}(\zeta, \eta) = \int_M (\alpha^T\alpha)^{-1}_x(\zeta_x, \eta_x)\on{vol}(\alpha).
	\end{align}
In the following we will derive an expression of this metric in local coordinates
$\{x^i, i=1,\cdots,m\}$.  With respect to the corresponding basis $\{\frac{\partial}{\partial x^i}\}$ on $T_xM$ and the standard basis on $\RR^n$,
the one-forms $\alpha_x, \zeta_x$ and $\eta_x$ can be represented by $n\times m$ matrices, which we will still denote by $\alpha_x, \zeta_x$ and $\eta_x$. Furthermore, the metric $(\alpha^T\alpha)_x$ can be identified with the $m\times m$ matrix $\alpha_x^T\alpha_x$. Thus we obtain the local formula of the Riemannian metric~\eqref{metric} as:
	\begin{align}\label{eq.metricLocal}
	G_{\alpha}(\zeta, \eta) = \int_M \tr(\zeta_x(\alpha^T\alpha)^{-1}_x \eta_x^T)\sqrt{\det(\alpha^T\alpha)_x}dx.
	\end{align}
	It is easy to see that by definition our metric $G$ \eqref{metric} is independent of the original metric on $M$. In addition, it follows from the local formula  and the second invariance of Lemma~\ref{matmetinvar} that the metric $G$ does not depend on the choice of coordinates near $x\in M$. The following lemma gives two important invariances of our metric $G$ on $\Omega^1_+(M,\mathbb R^n)$.

 \begin{lem}\label{lem:invariances}
Let $\alpha\in\Omega^1_+(M,\mathbb R^n)$ and $\zeta, \eta\in T_\alpha\Omega^1_+(M,\mathbb R^n)$.
\begin{enumerate}
\item The  metric \eqref{metric} is invariant under pointwise left multiplication with $\on{O}(n)$, i.e., for any smooth function $z: M\to\on{O}(n)$, we have
\begin{equation*}
G_{\alpha}(\zeta, \eta)=G_{z\alpha}(z\zeta,z\eta)
\end{equation*}
\item The  metric \eqref{metric} is invariant under the right action of the diffeomorphism group, i.e., for any $\varphi \in \Diff(M)$ we have
\begin{equation*}
G_{\alpha}(\zeta, \eta)=G_{\varphi^*\alpha}(\varphi^*\zeta,\varphi^*\eta)
\end{equation*}
\end{enumerate}
\end{lem}
\begin{proof}
The proof of the first invariance property is the same as for the finite dimensional metric on $M_+(n,m)$ from Lemma \ref{matmetinvar}. For the second invariance property we calculate
\begin{align*}
&G_{\varphi^*\alpha}(\varphi^*\zeta,\varphi^*\eta)= \int_M  \tr \left( \varphi^*\zeta\;((\varphi^*\alpha)^T\varphi^*\alpha)^{-1}(\varphi^*\eta)^T \right) \sqrt{\operatorname{det}\left((\varphi^*\alpha)^T\varphi^*\alpha\right)}\; \mu\\
&=\int_M  \tr \left(\zeta\circ\varphi\;((\alpha\circ\varphi)^T\alpha\circ\varphi)^{-1}(\eta\circ\varphi)^T \right) \sqrt{\operatorname{det}\left((\alpha\circ\varphi)^T\alpha\circ\varphi \right)}|\det(d\varphi)|\; \mu\\
&=G_{\alpha}(\zeta, \eta)
\qedhere
\end{align*}
\end{proof}

\subsection{Connection to the Ebin metric}\label{Sect:submersiontoEbin}
In this section we will show that the metric defined in \eqref{metric} on the space $\Omega_+^1(M, \mathbb R^n)$ is connected to the Ebin metric on the space of Riemannian metrics $\operatorname{Met}(M)$ on $M$. This will be a consequence of the previous result for the finite dimensional spaces of matrices and the point-wise nature of the metric.
The main difficulty in the infinite-dimensional situation is proving the surjectivity of the projection map.

Following \cite{ebin1970manifold} we will first recall the definition of the Ebin metric.
The space of Riemannian metrics $\operatorname{Met}(M)$ is a open subset of the space of all smooth symmetric $(0,2)$ tensor fields on $M$, denoted by $\Gamma(S^2T^*M)$, and thus the tangent space at each element $g$ is $\Gamma(S^2T^*M)$ itself. Let $g\in \operatorname{Met}(M)$ and $h, k\in T_g{\operatorname{Met}(M)}=\Gamma(S^2T^*M)$.
We can then introduce the metric via
\begin{align}\label{eq.metric.metricsOnM}
(h, k)_g = \dfrac14\int_M\tr_g(hk)\mu_g,
\end{align}
where $\mu_g$ is the volume form induced by $g$ and
where at any $x\in M$, we define the integrand by replacing $g, h, k$ by the associated symmetric $m\times m$ matrices $g(x), h(x), k(x)$ with
respect to an arbitrary basis of $T_xM$ and where we define $$\tr_g(hk)(x)=\tr(h(x)g(x)^{-1}k(x)g(x)^{-1}).$$

Recall that in Section~\ref{sec:symmetric} we have shown that the mapping
$\pi: M_+(n,m)\to\operatorname{Sym}_+(m)$, $
\pi(a)=a^Ta$
is a Riemannian submersion, where the metric on $M_+(n,m)$ is given by \eqref{eq.metric.a} and the metric on $\operatorname{Sym}_+(m)$ is given by \eqref{eq.metric.sym}. Similarly, we can define a mapping
\begin{align}\label{induced_metric_oneform}
\tilde{\pi}: \Omega_+^1(M, \mathbb R^n) \to \operatorname{Met}(M),\quad \alpha \mapsto \alpha^T \alpha,
\end{align}
where $(\alpha^T\alpha)$ is a section of $S^2_*TM$ which for $x\in M$ (the pullback of the Euclidean scalar product under $\alpha$.
We have the following result:
\begin{thm}
Let $M$ and $n$ be such that there exists at least one full-ranked $\mathbb R^n$ valued one-form on $M$, i.e., $\Omega^1_+(M,\mathbb R^n)\neq \emptyset$.
Then the mapping $\tilde{\pi}: \Omega_+^1(M, \mathbb R^n) \to \operatorname{Met}(M)$ is a Riemannian submersion, where $\Omega_+^1(M, \mathbb R^n)$ is equipped with the metric \eqref{metric} and $\operatorname{Met}(M)$ carries the multiple of the Ebin metric, as defined in \eqref{eq.metric.metricsOnM}.
\end{thm}
\begin{proof}
We first need to show that $\tilde \pi$ is a surjective map, i.e., given $g\in  \operatorname{Met}(M)$ we need to construct  $\beta(x)\in \Omega_+^1(M, \mathbb R^n)$
with $\tilde \pi(\beta)=g$. Therefore let $\al_0 \in \Omega_+^1(M, \mathbb R^n)$ be any fixed full-ranked one-form and let $g_0$ be the Riemannian metric induced by $\al_0$
via pulling back the Euclidean scalar product, see \eqref{induced_metric_oneform}.

%
%
Then for each $x\in M$, the operator $Y_x = g_x(g_0)_x^{-1}$ from $T_xM$ to itself,
defined by $g_0(Y(u),v) = g(u,v)$ for $u,v\in T_xM$, is positive-definite and 
symmetric with respect to the Riemannian metric $(g_0)_x$. Since $g$ and $\alpha_0$ are smooth tensor fields, 
$Y_x$ depends smoothly on $x$. The pointwise positive-definite square root $\sqrt{Y_x}$ is uniquely determined,
and it is a smooth function of $x$ as well (see Kato, Perturbation Theory, II.6~\cite{Kato}). We then define 
$\beta_x = (\alpha_0)_x\circ \sqrt{Y_x}$, which is again smooth in $x$ and maps each $T_xM$ to $\mathbb{R}^n$.
We verify that 
$$ \langle \beta(u), \beta(v)\rangle_{\mathbb{R}^n} = g_0\big(\sqrt{Y}(u),\sqrt{Y}(v)\big) = g_0(\sqrt{Y}{}^T \sqrt{Y}(u), v) = g(u,v)$$
for all $u,v\in T_xM$, so that $\beta^T \beta = g$. 
It follows that $\pi(\beta)=g$.
Since the metric on $\Omega_+^1(M, \mathbb R^n)$ and the metric on $ \operatorname{Met}(M)$ are both point-wise, the remainder of the result is now an immediate consequence of Theorem~\ref{thm:matrices:submersion}.
\end{proof}

\subsection{A product structure for the space of one-forms}
We begin this section by fixing a volume form $\mu$ on $M$. Whenever we refer to a matrix operation on a 1-form (e.g., trace or transpose), it is assumed that we have expressed that form locally as a matrix field, using a basis of the tangent space that has unit volume with respect to $\mu$.

Following the work of \cite{freed1989basic} we will decompose the space of 1-forms as the product
of the space of volume forms on $M$ with the space of 1-forms that induce the fixed volume form $\mu$, i.e.,
$\Omega_+^1(M, \mathbb R^n)\equiv \operatorname{Vol}(M)\times \Omega^1_{\mu}(M, \mathbb R^n)$, where $\Omega^1_{\mu}(M, \mathbb R^n)$ denotes the set of all 1-forms such that $\operatorname{det}\left(\alpha^T\alpha \right)=1$. A straight-forward calculation shows that the tangent space of $\Omega^1_{\mu}(M, \mathbb R^n)$ consists of all tangent vectors $h\in T_{\alpha}\Omega^1_{\mu}(M, \mathbb R^n)$ such that
$\on{tr} (\alpha^+ h)=0$ with $\alpha^+=(\alpha^T \alpha)^{-1}\alpha^T$ being the Moore-Penrose pseudo-inverse. In the following lemma we calculate the formula of the metric $G$
in this product decomposition:
\begin{lem}
In the identification $\Omega_+^1(M, \mathbb R^n)\equiv \operatorname{Vol}(M)\times \Omega_{\mu}^1(M, \mathbb R^n)$ the metric \eqref{metric} takes the form
\begin{equation}
\bar G_{(\rho,\beta)}\left((\nu_1,h_1),(\nu_2,h_2)\right)=
\int_M       \tr \left(h_1\;(\beta^T\beta)^{-1}h_2^T \right)         \rho\mu+\frac1{m}
\int_M       \frac{\nu_1}{\rho}\frac{\nu_2}{\rho}        \rho\mu\\
\end{equation}
The metric $\bar G$ is not a product metric, since the foot-point volume density $\rho$ appears in both terms above. Note, however, that the decomposition of the tangent space into directions tangent to $\on{Vol}(M)$ and directions tangent to $\Omega_{\mu}(M, \mathbb R^n)$ are orthogonal to each other with respect to the metric $\bar G$. Such a metric is also called an almost product metric, see~\cite{gil1992pseudoriemannian}.
\end{lem}
\begin{proof}
We first construct a bijection from $\operatorname{Vol}(M)\times \Omega^1_{\mu}(M, \mathbb R^n)$ to the space of full-ranked one-forms. Therefore we let
\begin{equation}
\Phi(\alpha):= (\rho,\beta)=\left(\sqrt{\operatorname{det}(\alpha^T\alpha)}, \rho^{-1/m}\alpha  \right)\qquad \Phi^{-1}(\rho,\beta) := \rho^{1/m} \beta\;.
\end{equation}
To see that this mapping has the required properties, we calculate
\begin{align*}
\sqrt{\operatorname{det}(\beta^T\beta)}=  \rho^{-1} \sqrt{\operatorname{det}(\alpha^T\alpha)}=1\;.
\end{align*}
To calculate the induced metric on the product we have to calculate the variation of the inverse mapping.
We have
\begin{equation}
d\Phi^{-1}(\rho, \beta)(\nu,h)= \rho^{1/m} h+\frac1m\rho^{1/m-1}\nu\beta
\end{equation}
Thus we obtain the formula of the metric on the product space:
\begin{align*}
&\bar G_{(\rho,\beta)}\left((\nu_1,h_1),(\nu_2,h_2)\right)\\&\qquad=
G_{\Phi^{-1}(\rho,\beta)}\left( d\Phi^{-1}(\rho, \beta)(\nu_1,h_1),d\Phi^{-1}(\rho, \beta)(\nu_2,h_2)\right)\\
&\qquad=
G_{\rho^{1/m} \beta}\left(\rho^{1/m} h_1+\frac1m\rho^{1/m-1}\nu_1\beta,\rho^{1/m} h_2+\frac1m\rho^{1/m-1}\nu_2\beta\right)\\
&\qquad=
\int_M       \tr \left(h_1\;(\beta^T\beta)^{-1}h_2^T \right)         \rho\mu+\frac1{m}
\int_M       \frac{\nu_1}{\rho}\frac{\nu_2}{\rho}        \rho\mu\\&\qquad\qquad\qquad+\frac{v_2}{m}
\int_M       \tr \left(h_1\;(\beta^T\beta)^{-1}\beta^T \right) \mu
+
\frac{v_1}{m}\int_M       \tr \left(\beta\;(\beta^T\beta)^{-1}h_2^T \right) \mu
\end{align*}
Now the result follows since any tangent vector $h$ to $\Omega^1_{\mu}$ satisfies
$$\tr \left(h\;(\beta^T\beta)^{-1}\beta^T \right)=0.$$
Note that, by  standard properties of the trace, this also shows that last term vanishes.
\end{proof}
\begin{rem}
If one restricts the metric to the space of volume forms $\on{Vol}(M)$ one obtains the Fisher-Rao metric.
For this metric the geometry is well-studied and completely understood, see e.g., \cite{friedrich1991fisher,khesin2013geometry}. Furthermore, it has been
shown that the Fisher-Rao metric is up to a constant the unique Riemannian metric on the space of volume densities that is invariant under the action of the diffeomorphism group \cite{ay2015information,bauer2016uniqueness,cencov2000statistical}.
\end{rem}

\subsection{The geodesic distance}
Any Riemannian metric (on a finite or infinite dimensional manifold) gives rise to a (pseudo) distance on the manifold, the geodesic distance. In finite dimensions this distance function is always a true metric, i.e.,
symmetric, satisfies the triangle inequality and non-degenerate. In infinite dimensions it has been shown that the third property might fail, see \cite{eliashberg1993biinvariant,michor2005vanishing,bauer2013geodesic,bauer2018vanishing}. In this section we will observe that the geodesic  distance function of the metric \eqref{metric} can be written as an integral over the geodesic distance function of a finite dimensional space of matrices and thus we will obtain the non-degeneracy of the geodesic distance on the infinite dimensional space of all full ranked one-forms. This is essentially the same proof as for the Ebin-metric on the space of all Riemannian metrics; see the work of Clarke \cite{clarke2013geodesics}.

To formulate this result we recall the finite dimensional Riemannian metric on the space $M_+(n,m)$:
\begin{equation}
\langle u,v\rangle_a= \tr \left(u\;(a^Ta)^{-1}v^T \right) \sqrt{\operatorname{det}\left(a^Ta \right)}\;.
\end{equation}
Furthermore we denote the corresponding geodesic distance by $\operatorname{dist}^{n\times m}(\cdot,\cdot)$. Note that $\operatorname{dist}^{n\times m}$ is non-degenerate as
the space of $n\times m$ matrices is finite dimensional.

With this notation we immediately obtain the following result concerning the geodesic distance on the infinite dimensional manifold
of all full-ranked one-forms:
\begin{thm}\label{thm:distpointwise}
The geodesic distance on the manifold $\Omega^1_+(M,\mathbb R^n)$ is non-degenerate and satisfies
\begin{equation}\label{distance:pointwise}
\operatorname{dist}^{\Omega^1_+}(\alpha,\beta)^2 \geq \int_M \operatorname{dist}^{n\times m}(\alpha(x),\beta(x))^2\; \mu\;.
\end{equation}
\end{thm}
\begin{proof}
To prove this result we only need to show the inequality \eqref{distance:pointwise}.
The non-degeneracy of the geodesic distance  follows then directly from the non-degeneracy of the geodesic distance on finite dimensional manifolds
and the face that two distinct elements of $\Omega^1_+(M,\mathbb R^n)$  have to differ on a set of positive measure.
The proof of the above inequality is exactly the same as in \cite[Thm. 2.1]{clarke2013geodesics}
\end{proof}
\begin{rem}
For the Ebin metric on the space of all Riemannian metrics it has been shown
that the analogue of the inequality~\eqref{distance:pointwise} is actually an equality, i.e.,
that
\begin{equation}\label{distance:pointwise:eq}
\operatorname{dist}^{\on{Met}}(\alpha,\beta)^2 = \int_M \operatorname{dist}^{m\times m}(\alpha(x),\beta(x))^2\; \mu\;.
\end{equation}
It is easy to generalize this result to the situation studied here by allowing paths of one-forms that
are only of class $L^2$ in $x\in M$. Therefore one simply chooses for each $x\in M$
a short path in the finite dimensional manifold $\mathbb R^{n\times m}$, which immediately yields the equality. Here a short path means a path of matrices $a(t)$ such that
$\operatorname{len}(a(t))\leq \operatorname{dist}^{n\times m}(a(0),a(1))+\epsilon$ for some $\epsilon >0$. To prove the result in the smooth category is much harder.
We believe, however, that a similar analysis as in  \cite{clarke2013geodesics} might be used to obtain this result. We leave this question open for future research.
\end{rem}

\subsection{Geodesics and curvature}\label{sec:geodesicequation_forms}
The point-wise nature of the metric will allow us to directly use our results for the space of matrices to obtain the following result concerning geodesics and curvature, c.f. \cite{Misiolek1993Stability,Bao1993nonlinear}.

\begin{thm}\label{thm.SectionalCurvatureandGeodesics_forms}
The geodesic equation of the generalized Ebin metric on the space of full-ranked one-forms decouples in space and time. Thus for each $x\in M$ it is given by the ODE
\eqref{eq.geodesic.a} with explicit solution as presented in Theorem~~\ref{thm.geodesicformula.a}. Similarly, the sectional curvature is simply the integral over the pointwise sectional
curvatures and thus the statements on sign-definiteness of Theorem~\ref{thm.SectionalCurvature} hold also in this infinite dimensional situation.
\end{thm}

\subsection{On totally geodesic subspaces}
In this section we will show that the space $\Omega^1_+(M,\mathbb R^n)$ contains two remarkable totally geodesic subspaces. To understand one of these subspaces, we need some preliminaries. Let $\hbox{Gr}(m,n)$ denote the Grassmannian manifold of all $m$-dimensional linear subspaces of $\mathbb R^n$. Define a map
$$W:\Omega^1_+(M,\mathbb R^n)\to C^\infty(M,\hbox{Gr}(m,n))$$ by
$$W(\alpha)(x)=\alpha(T_xM).$$
Let $\xi$ denote the canonical $m$-plane bundle over $\hbox{Gr}(m,n)$.
Given $f\in C^\infty(M,\hbox{Gr}(m,n))$, it is easy to see that $f\in W(\Omega^1_+(M,\mathbb R^n))$ if and only if $TM\cong f^*(\xi)$. This is because $W(\alpha)=f$ if and only if $\alpha$ is a bundle isomorphism $TM\to f^*(\xi)$.
\begin{thm}\label{thm.totallyGeodesicSubspace.2}
The following spaces are totally geodesic subspaces of the space $\Omega^1_+(M,\mathbb R^n)$ equipped with the generalized Ebin metric:
\begin{enumerate}
\item any one-dimensional space of scalings $\mathcal A :=\left \{ t\alpha_0 \,|\,t\in \mathbb R_{>0}\right\}$, where $\alpha_0$ is a fixed element of  $\Omega^1_+(M,\mathbb R^n)$,
\item the space $\mathcal B :=\left \{\alpha \in \Omega^1_+(M,\mathbb R^n) | W(\alpha)=f_0\right\},$ where $f_0$ is any fixed element of $C^\infty(M,\hbox{Gr}(m,n))$. (Note that this space is empty unless $TM\cong f_0^*(\xi)$, by the remark just above this Lemma).
\end{enumerate}

\end{thm}

\begin{proof}
	Here we use the point-wise nature of the metric \eqref{metric} on the space $\Omega_+^1(M, \mathbb R^n)$. Let $x\in M$ and $\{e_i, 1\leq i\leq m\}$ be an orthonormal basis of $T_xM$. Choosing the standard basis for $\mathbb R^n$, (1) follows immediately from the first statement of Theorem~\ref{thm.totallyGeodesicSubspace.M}.
	
	Now we prove (2), i.e., the space $\mathcal B$ is a totally geodesic subspace. Since $\alpha$ is a bundle isomorphism $TM\to f_0^*(\xi)$, for each $x\in M$ the image of the orthonormal basis under $\alpha_x$, denoted by $\{\tilde{e}_i = \alpha_x(e_i)\}$, forms an orthonormal basis of $\xi_{f_0(x)}$. Note that $\xi_{f_0(x)} = f_0(x)$ is a $m$-plane. So we can extend this orthonormal basis to get an orthonormal basis $\{\tilde{e}_i, 1\leq i\leq n\}$ of $\mathbb R^n$. With respect to this basis $\{e_i\}$ of $T_xM$ and the basis $\{\tilde{e_i}\}$ of $\mathbb R^n$,
	it is easy to see that each linear transformation in $\left\{\alpha_x: T_xM\to \mathbb R^n\,|\, W(\alpha)(x)=f_0(x)\right\}$ corresponds to a matrix in $\operatorname{GL}(m)$ (extended to a $n\times m$ matrix with zeros). Thus the result follows from the second statement of Theorem~\ref{thm.totallyGeodesicSubspace.M}.
\end{proof}

\subsection{Metric and geodesic incompleteness}
As a consequence of the fact that scaling of a full ranked one-form is totally geodesic, we immediately obtain the geodesic and metric
incompleteness of the metric:
\begin{thm}
The space  $\Omega^1_+(M,\mathbb R^n)$ is metrically and geodesically incomplete.
\end{thm}
\begin{proof}
This follows directly from the fact that scaling  of a metric yields geodesic curves that leave the space in finite time, c.f. Theorem~\ref{thm.totallyGeodesicSubspace.2}.
\end{proof}
To obtain the metric completion we believe that a similar strategy as  in \cite{clarke2013completion} will lead to the following result:
\begin{conj}
The metric completion of the space $\Omega^1_+(M,\mathbb R^n)$ equipped with the geodesic distance function of the generalized Ebin metric is the  space of $L^2$-sections
of  the vector bundle  $\left(T^*M\otimes \mathbb R^n\right) \to M$ modulo the equivalence relation $\sim$,
where $\alpha \sim \beta$ if the statement
\begin{equation*}
\alpha(x)\neq \beta(x) \Longleftrightarrow \operatorname{rank}(\alpha(x)) < m \text{ and } \operatorname{rank}(\beta(x)) < m
\end{equation*}
holds almost everywhere.
\end{conj}
The proof of \cite{clarke2013completion} used  rather heavy machinery from geometric measure theory. To develop this theory in the current context is out of the scope of the present article.
Thus we  leave this question open for future research.

\section{An application: Reparametrization invariant metrics on the space of open curves}\label{shape:analysis}
In this section we will describe the relation of our proposed metric to the square root framework as developed for shape analysis
of curves \cite{younes1998computable,younes2008metric,Srivastava2011Shape}. In contrast to the aforementioned framework, our construction is not limited to one-dimensional objects, but has a direct
generalization to higher dimensional objects, notably to the space of surfaces. We plan to develop this line of research in a future
application-oriented article and will focus mainly on the simpler space of curves in this section.

In the following we denote the space of immersed curves in $\mathbb R^n$ by
\begin{equation}
\operatorname{Imm}([0,1],\mathbb R^n):=\left\{c\in C^{\infty}([0,1],\mathbb R^n): |c'|\neq 0 \right\}\;.
\end{equation}
Here $c'$ denotes the derivative of $c$ with respect to $\theta \in [0,1]$.
We can map each curve to a $\mathbb R^n$-valued one-form on $[0,1]$ via
$c \mapsto c' d\theta$. The immersion condition ensures that
the resulting one-form actually has full rank and thus we obtain
a bijection
\begin{equation}\label{eq.differential.Phi}
\Phi: \operatorname{Imm}([0,1],\mathbb R^n)/\operatorname{trans} \to \Omega^1_+([0,1],\mathbb R^n)\;.
\end{equation}
To see that the map $\Phi$ is surjective, note that all one-forms are closed since we are in dimension one, and all closed one-forms are exact since the first cohomology of $[0,1]$ vanishes.
Furthermore, we had to identify curves that differ only by a translation as they all get mapped to the same one-form.
Pulling back the metric~\eqref{metric} on  $\Omega^1_+([0,1],\mathbb R^n)$, one obtains a reparametrization invariant
metric on the space of curves modulo translations. It turns out that this metric is exactly the Younes-metric as studied in
\cite{younes1998computable}:
\begin{equation}\label{eq.metric.curve}
 (\Phi^*G)_c(h,k)= G_{\Phi(c)}\left(d\Phi(c)(h),d\Phi(c)(k)\right)=\int_0^1 \frac{h_{\th}\cdot k_\th}{|c'|}d\theta = \int_0^1 D_sh\cdot D_sk ds,
\end{equation}
where $c\in \operatorname{Imm}([0,1],\mathbb R^n)$ and $h,k \in T_c \operatorname{Imm}([0,1],\mathbb R^n)$. Here
$D_s=\frac{1}{|c'|} \frac{d}{d\theta}$ denotes arc-length differentiation, and $ds=|c'|d\theta$ denotes integration with respect to arc length.
In the article \cite{younes2008metric} the authors introduced a transformation for this metric that yields explicit formulas for geodesics between open and closed curves in the plane.
Implicitly this has been extended to open curves in arbitrary dimension in the article \cite{needham2018shape}.
By considering the formulas of Section~\ref{sec:geodesicequation_forms} in the special case studied in this section, we obtain an explicit formula for geodesics for curves in arbitrary dimension, and in addition we obtain the non-negativity of the sectional curvature:
\begin{thm}
Let $c_0\in  \operatorname{Imm}([0,1], \mathbb R^n)$ and $h\in T_{c_0}\operatorname{Imm}([0,1], \mathbb R^n)$.
The geodesic on the space of open curves modulo translations $\operatorname{Imm}([0,1], \mathbb R^n)/\operatorname{trans}$ starting at $c_0$ in the direction $h$ with respect to the metric \eqref{eq.metric.curve} is given by
\begin{align}\label{eq.geodesic.curve}
c(t, \theta) = \int_0^{\theta}f(t,\lambda)e^{-s(t,\lambda)\left(V^T(\lambda)-V(\lambda)\right)}c'_0(\lambda)d\lambda,
\end{align}
where
\begin{align*}
&V(\theta) = h(\theta)\,c_0'^+(\theta),\qquad \delta_0(\theta) = \operatorname{tr}(V^T(\theta)V(\theta)),\qquad \tau_0(\theta) = \operatorname{tr}(V(\theta)),\\
&f(t,\theta) = \dfrac{\delta_0(\theta)}{4} t^2 + \tau_0(\theta) t + 1,\qquad s(t,\theta) = \int_0^t 1/f(\sigma,\theta) {d\sigma}\,.
\end{align*}
Furthermore, the sectional curvature of $\operatorname{Imm}([0,1], \mathbb R^n)/\operatorname{trans}$ with respect to the metric \eqref{eq.metric.curve} is always non-negative for $n\geq2$ and vanishes for $n = 2$.
\end{thm}
\begin{proof}
To prove the statement on the explicit solution we consider the formula given in Theorem~\ref{thm.geodesicformula.a} for $m = 1$.
Let $c(t,\theta)$ be the geodesic starting at $c_0 = c(0, \theta)$ in direction $h$. We will use $c'$ to denote the derivative with respect to $\theta$ and $c_t$ to denote the derivative in time $t$.
Using the notation $V(\theta) = c'_t(0,\theta)c'^+(0,\theta) = hc_0'^+$
we have:
	\begin{align*}
	c_{\theta}(t,\theta) = f(t,\theta)e^{-s(t,\theta)\omega_0}c'_0(\theta)e^{s(t,\theta)P_0},
	\end{align*}
	where $f(t,\theta)$ and $s(t,\theta)$ are as in Theorem~\ref{thm.geodesicformula.a}, $\omega_0 = V^T(\theta) - V(\theta)$ and 
	$$P_0 = \left(c_{\theta}^Tc_{\theta}\right)^{-1}v^Tc_{\theta}-\tau_0 =0\;,$$
	Taking the integral with respect to $\theta$ formula \eqref{eq.geodesic.curve} follows. 
	The result on the sectional curvature follows directly from statement \eqref{sec:curvature:m1}
	of Theorem~\ref{thm.SectionalCurvature} and Theorem~\ref{thm.SectionalCurvatureandGeodesics_forms}.
\end{proof}
In Figure~\ref{CurvesExamples}, we present one example of a geodesic that was computed using the explicit formula derived above.

\subsection{The space of surfaces}

In this section we will briefly comment on the difficulties that arise for using the same method to obtain a framework for shape analysis  of surfaces. As mentioned above, in the case of curves, the mapping $\Phi$ in \eqref{eq.differential.Phi} gives us a bijection between the space of curves modulo translations and the space of full rank $\RR^n$-valued one-forms on $[0,1]$. Thus the preimage of a geodesic in the space $\Omega_+^1([0,1], \mathbb R^n)$ gives a geodesic in the space of immersed curves in $\mathbb R^n$. However, in the case of (two-dimensional) surfaces in $\mathbb R^3$ (here typically $n$ will be $3$), the operator $d: \operatorname{Imm}(S^2,\mathbb R^3)/\operatorname{trans}\to \Omega_+^1(S^2, \mathbb R^3)$ only induces a bijection between $\operatorname{Imm}(S^2,\mathbb R^3)/\operatorname{trans}$ and the space of full rank and \emph{exact} one forms, denoted by ${\Omega}_{+,\on{ex}}^1(S^2, \mathbb R^3)$, which is a proper subspace of $\Omega_+^1(S^2, \mathbb R^3)$.
Furthermore ${\Omega}_{+,\on{ex}}^1(S^2, \mathbb R^3)$  is not a totally geodesic submanifold of $\Omega_+^1(S^2, \mathbb R^3)$ and  so geodesics in $\Omega_+^1(S^2, \mathbb R^3)$ do not give rise to geodesics in $\operatorname{Imm}(S^2,\mathbb R^3)/\operatorname{trans}$. Note that the same would be true for $S^2$ replaced with the sheet $[0,1]\times [0,1]$.

Thus using this representation for shape analysis of surfaces will require some extra work. A potential approach is to study the submanifold geometry of ${\Omega}_{+,\operatorname{ex}}^1(S^2, \mathbb R^3)$ in more detail to obtain an explicit solution in this space. Alternatively one could work in the space of all full rank one-forms $\Omega_+^1(S^2, \mathbb R^3)$ and project the geodesic onto the submanifold ${\Omega}_{+,\operatorname{ex}}^1(S^2, \mathbb R^3)$. In Figure~\ref{SurfaceExample1} and~\ref{SurfaceExample2}, we present examples of geodesics between two parametrized surfaces with respect to the pull-back metric, that have been calculated using a discretization of the space of full ranked exact one-forms.  These examples have been calculated using the numerical framework for the Riemannian metric studied in this paper as developed in~\cite{su2019shape}, where the spherical parametrizations of the boundary surfaces have been obtained using the code of Laga et al.~\cite{kurtek2013landmark}.
\begin{figure*}
	\begin{center}
		\includegraphics[width=0.9\linewidth]{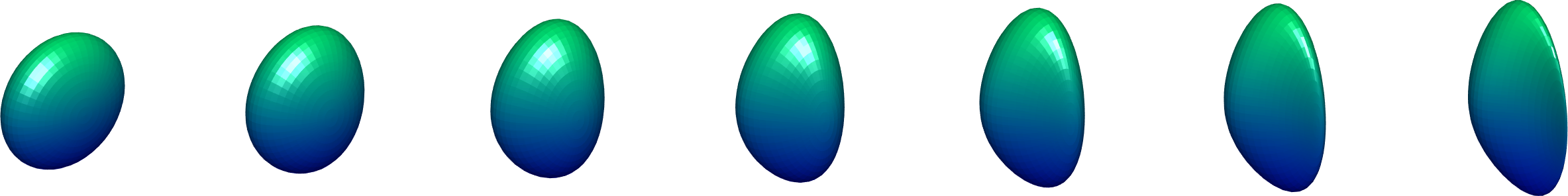}
	\end{center}
	\caption{A geodesic in the space of surfaces modulo translations with respect to the generalized Ebin metric~\eqref{metric}.}
	\label{SurfaceExample2}
\end{figure*}

\appendix
\section{The computation of the geodesic formula in the space $M_+(n,m)$}\label{appendix.A}

In this appendix we give the computation of the geodesic formula in the space $M_+(n,m)$ with respect to the metric \eqref{eq.metric.a}. Recall that the geodesic equation on $M_+(n,m)$ is given by
\begin{equation}
\label{eq.appendix.geodesic.a}
\begin{aligned}
a_{tt} = &a_t(a^Ta)^{-1}a_t^Ta + a_t(a^Ta)^{-1}a^Ta_t
-a(a^Ta)^{-1}a_t^Ta_t\\
&\qquad\qquad\qquad+\frac12\tr\left(a_t(a^Ta)^{-1}a_t^T\right)a-\tr\left(a_t(a^Ta)^{-1}a^T\right)a_t,
\end{aligned}\end{equation}
and a simpler form of the geodesic equation for $L = a_ta^+$ is given by
\begin{align}\label{eq.appendix.geodesic.L}
L_t + \tr(L)L + (L^TL-LL^T) - \frac{1}{2} \tr(L^TL) aa^+ = 0.
\end{align}
To solve the equation \eqref{eq.appendix.geodesic.a}, we start with the equation \eqref{eq.appendix.geodesic.L} for $L(t)$ and we have the following proposition.
\begin{prop}\label{prop.appendix.taudelta}
	Suppose $a$ and $L$ are as in \eqref{eq.appendix.geodesic.a} and \eqref{eq.appendix.geodesic.L}. Define $\delta = \tr(L^TL)$ and $\tau = \tr(L)$. Then
	$\tau$ and $\delta$ satisfy the differential equations
	\begin{align}
	\begin{cases}\label{eq.appendix.taudelta}
	\tau_t + \tau^2 - \tfrac{m}{2} \delta = 0, \qquad &\tau(0)=\tau_0 = \tr(L(0))\\
	\delta_t + \tau \delta = 0, \qquad &\delta(0)=\delta_0 = \tr(L(0)^TL(0)).
	\end{cases}
	\end{align}
	The solution of these equations is
	\begin{equation}\label{eq.appendix.taudelta.solutions}
	\tau(t) = \frac{f_t(t)}{f(t)}, \qquad \delta(t) = \frac{\delta_0}{f(t)},
	\end{equation}
	where
	\begin{equation}\label{eq.appendix.f}
	f(t) = \frac{m\delta_0}{4} t^2 + \tau_0 t + 1.
	\end{equation}
\end{prop}
\begin{proof}
	The trace of \eqref{eq.appendix.geodesic.L} yields the first equation in  \eqref{eq.appendix.taudelta} since $\tr{(aa^+)} = \tr(a^+a) = \tr(I_{m\times m}) = m$. Notice that $Laa^+ = L$. We have
	\begin{align*}
	\tr(L^Taa^+) &= \tr(aa^+L^T) = \tr((Laa^+)^T) =  \tr(L^T) = \tr(L).
	\end{align*}
	Multiplying \eqref{eq.appendix.geodesic.L} on the left by $L^T$
	yields the second equation in \eqref{eq.appendix.taudelta}.
	The system \eqref{eq.appendix.taudelta} is exactly the same system as in the work of Freed-Groisser~\cite{freed1989basic}. Thus we can use the same trick to solve it. Write $\tau(t) = f_t(t)/f(t)$ where $f(0) = 1$, and the first equation in \eqref{eq.appendix.taudelta} becomes
	\begin{align}
	f_{tt}(t) = \dfrac{m}{2}\delta(t)f(t),\quad f(0) = 1,\, f_t(0) = \tau_0.
	\end{align}
	Meanwhile the second equation in \eqref{eq.appendix.taudelta} becomes $\dfrac{d}{dt}(\delta f) = 0$, which can immediately be solved to give $\delta(t) = \delta_0/f(t)$.
	So the second derivative $f_{tt}(t) = \dfrac{m\delta_0}{2}$ is constant, and with $f_t(0)=\tau_0$ and $f(0) = 1$, we get the solution $f(t) = \dfrac{m\delta_0}{4} t^2 + \tau_0 t + 1.$ Formula \eqref{eq.appendix.taudelta.solutions} follows.
\end{proof}

With explicit solutions for $\tau(t)$ and $\delta(t)$ in hand, we can now solve the rest of the geodesic equation \eqref{eq.appendix.geodesic.L} with initial $L(0)$, given by
\begin{equation} \label{eq.appendix.geo.L}
L_t + \frac{f_t}{f} L + (L^TL-LL^T) - \frac{\delta_0}{2f}aa^+ = 0.
\end{equation}

\begin{lem}
	Let $M(t) = L(t) - \dfrac{\tau(t)}{m}a(t)a^+(t)$. Then $L$ satisfies \eqref{eq.appendix.geo.L} if and only if $M$ satisfies
	\begin{align}\label{eq.appendix.geo.M}
	M_t + \frac{f_t}{f}M + (M^TM-MM^T) = 0.
	\end{align}
\end{lem}
\begin{proof}
	We first compute
	\begin{align*}
	(aa^+)_t =& \left(a(a^Ta)^{-1}a^T\right)_t\\
	=& a_t(a^Ta)^{-1}a^T -a(a^Ta)^{-1}\left(a^T_ta+a^Ta_t\right)(a^Ta)^{-1}a^T + a(a^Ta)^{-1}a^T_t\\
	=& L - L^Taa^+ - aa^+L + L^T\\
	=& M - M^Taa^+ - aa^+M +M^T\;.
	\end{align*}
	Here we used that $L = a_ta^+ = a_t(a^Ta)^{-1}a^T$,  that $\tau_t = \dfrac{m}{2}\delta - \tau^2$ and  that $Maa^+ = M$. Thus we obtain
	\begin{align*}
		L_t =& M_t + \dfrac{\delta_0}{2f}aa^+ - \dfrac{\tau^2}{m}aa^+ + \dfrac{\tau}{m}\left(M - M^Taa^+ -aa^+M+ M^T\right),\\\
		\dfrac{f_t}{f}L =& \dfrac{f_t}{f}M + \dfrac{\tau^2}{m}aa^+,\\
		L^TL - LL^T =& M^TM -MM^T + \frac{\tau}{m}aa^+M + \frac{\tau}{m}M^Taa^+ - \frac{\tau}{m}M^T - \frac{\tau}{m}M.
	\end{align*}
	Replacing the terms in \eqref{eq.appendix.geo.L} with the formulas above we obtain equation \eqref{eq.appendix.geo.M} and thus the statement follows.
\end{proof}

\begin{prop}
	The solution of \eqref{eq.appendix.geo.L} satisfies
	\begin{equation}\label{eq.appendix.L}
	L(t) = \frac{1}{f(t)}e^{-s(t)\omega_0}M_0e^{s(t)\omega_0} + \frac{f_t(t)}{mf(t)}a(t)a(t)^+,
	\end{equation}
	where $\omega_0 = L(0)^T - L(0)$, $s(t) = \int_0^t \dfrac{d\sigma}{f(\sigma)}$ and
	$
	M_0 = L(0) - \dfrac{\tau_0}{m}a(0)a^+(0).
	$
\end{prop}
\begin{proof}
	Use equation \eqref{eq.appendix.geo.M} and set $M(t) = N(t)/f(t)$. Then $N$ satisfies
	\begin{equation*}
	N_t + \frac{1}{f}(N^TN- NN^T) = 0.
	\end{equation*}
	Changing variables to $s(t) = \int_0^t d\sigma/f(\sigma)$ we obtain
	\begin{equation}\label{eq.Nequation}
	N_s + N^TN - NN^T = 0.
	\end{equation}
	Note that the transpose of \eqref{eq.Nequation} is
	$ N^T_s + N^TN - NN^T = 0.$
	It follows that $\omega = N^T-N$ is constant in time, and thus $ \omega = \omega_0 = N^T(0) - N(0) = M^T(0) - M(0) = L^T(0)-L(0)$.
	We can rewrite \eqref{eq.Nequation} as
	\begin{align*}
	N_s = NN^T - N^TN = -\omega_0 N + N \omega_0 = [-\omega_0, N].
	\end{align*}
	Then we obtain the solution
	\begin{align}\label{eq.appendix.N}
	N(s) = e^{-s \omega_0} N(0) e^{s\omega_0}.
	\end{align}
	Translate \eqref{eq.appendix.N} back into
	\begin{equation}
	L(t) = M(t) + \frac{\tau}{m}a(t)a^+(t) = \frac{1}{f(t)} N(t) + \frac{\tau}{m}a(t)a^+(t),
	\end{equation}
	we obtain \eqref{eq.appendix.L}.
\end{proof}
Using formula \eqref{eq.appendix.L} of $L(t)$ we are now able to obtain a solution for the flow equation $a_t(t) = L(t)a(t)$.
\begin{thm}
	Let $f(t)$ be of the same form as in \eqref{eq.appendix.f}. Then the solution of the flow $a_t(t) = L(t)a(t)$ with initial data $a(0)$ is given by
	\begin{align}\label{eq.appendix.a}
	a(t) = f(t)^{1/m}e^{-s(t)\omega_0}a(0)e^{s(t)P_0},
	\end{align}
	where $\omega_0 = L^T(0)-L(0)$ and
	\begin{align*}
	P_0 = \left(a^T(0)a(0)\right)^{-1}a_t(0)^Ta(0) - \dfrac{\tau_0}{m}I_{m\times m}.
	\end{align*}
\end{thm}
\begin{proof}
	Using \eqref{eq.appendix.L}, the equation for $a(t)$ becomes
	\begin{align*}
	a_t = L a  = \dfrac{1}{f }e^{-s\omega_0}M_0e^{s\omega_0}a + \dfrac{f_t }{mf}a .
	\end{align*}
	Write $a(t) = f(t)^{1/m}Q(t)$ to eliminate the second term. Then we have
	\begin{align*}
	Q_t = \dfrac{1}{f}e^{-s\omega_0}M_0e^{s\omega_0}Q.
	\end{align*}
	Changing variables to $s(t) = \int_0^t\dfrac{d\sigma}{f(\sigma)}$ we obtain
	\begin{align*}
	Q_s = e^{-s\omega_0}M_0e^{s\omega_0}Q.
	\end{align*}
	Now let $Q(s) = e^{-s\omega_0}R(s)$. Then $R(s)$ satisfies the differential equation
	\begin{align*}
	R_s =& \omega_0R + M_0R = M_0^TR\\
	=& (L^T(0) - \dfrac{\tau_0}{m}a(0)a^+(0))R\\
	=& a(0)\left((a^T(0)a(0))^{-1}a_t^T(0) - \dfrac{\tau_0}{m}a^+(0)\right)R.
	\end{align*}
	Notice that the initial $R(0) = a(0)$ and $R_s$ is always of the form $a(0)$ times a $m\times m$ matrix. Therefore we must have $R(s) = a(0)B(s)$ for some $m\times m$ matrix $B$, which satisfies $B(0) = I_{m\times m}$ and
	\begin{align}\label{eq.appendix.B}
	B_s = \left(\left(a^T(0)a(0)\right)^{-1}a_t(0)^Ta(0) - \dfrac{\tau_0}{m}I_{m\times m}\right)B(s).
	\end{align}
	Let $P_0 = \left(a^T(0)a(0)\right)^{-1}a_t(0)^Ta(0) - \dfrac{\tau_0}{m}I_{m\times m}$. The solution of the equation \eqref{eq.appendix.B} with initial $B(0) = I_{m\times m}$ is
	\begin{align*}
	B(s) = e^{sP_0}.
	\end{align*}
	Changing back to $t$ variables, formula \eqref{eq.appendix.a} follows immediately.
\end{proof}

\section{The space of symmetric matrices (revisited)}\label{appendix:submersion}
In this appendix we re-derive some classical results by \cite{freed1989basic,gil1991riemannian,clarke2013completion,ebin1970manifold} concerning the (finite-dimensional version of the) Ebin-metric on the space of symmetric matrices using our Riemannian submersion picture.
We first present the geodesic equation on $\on{Sym}_+(m)$, which corresponds to the horizontal geodesic equation on $M_+(n,m)$:
\begin{cor}
	The geodesic equation on $\operatorname{Sym}_+(m)$ with respect to the metric \ref{eq.metric.sym} is given by
	\begin{align}
		g_{tt} = g_tg^{-1}g_t + \dfrac14\tr(g^{-1}g_tg^{-1}g_t)g - \dfrac12\tr(g^{-1}g_t)g_t.
	\end{align}
\end{cor}
\begin{proof}
	We identify the space of symmetric matrices $\operatorname{Sym}_+(m)$ with the quotient space $\operatorname{SO}(n)\backslash M_+(n,m)$ and consider the horizontal geodesic equation on $M_+(n, m)$, which is given by
	\begin{align}\label{eq.geodesicHorizontal}
		a_{tt} = a_ta^+a_t + \dfrac12\tr(a_t(a^Ta)^{-1}a_t^T)a - \tr(a_ta^+)a_t.
	\end{align}
	This is a straight-forward calculation using that $a_ta^+$ is symmetric.
	Now consider a smooth curve $g(t)$ in the space of symmetric matrices $\operatorname{Sym}_+(m)$. Then $g(t) = \pi(a(t)) = a(t)^Ta(t)$ for some horizontal lift $a(t)\in M_+(n,m)$ and
	\begin{align}
		g_t = a_t^Ta + a^Ta_t;\qquad g_{tt} = a_{tt}^Ta + 2a_t^Ta_t + a^Ta_{tt}.		
	\end{align}
	Inserting the expression of $a_{tt}$ in \eqref{eq.geodesicHorizontal} we obtain
	\begin{align}
	    g_{tt} =& a_{tt}^Ta + 2a_t^Ta_t + a^Ta_{tt}\\
		=& a_t^Ta(a^Ta)^{-1}a_t^Ta + 2a_t^Ta_t + a^Ta_t(a^Ta)^{-1}a^Ta_t\\
		&\quad + \tr(a_t(a^Ta)^{-1}a_t^T)a^Ta - \tr(a_ta^+)(a_t^Ta+a^Ta_t)
	\end{align}
	Notice that $a^+a = I$ and $a_ta^+$ is symmetric. It is easy to check that
	\begin{align}
		g_tg^{-1}g = a_t^Ta(a^Ta)^{-1}a_t^Ta + 2a_t^Ta_t + a^Ta_t(a^Ta)^{-1}a^Ta_t.
	\end{align}
	Similar to the calculation in Theorem~\ref{thm:matrices:submersion} we obtain
	\begin{align}
		\dfrac14\tr(g^{-1}g_tg^{-1}g_t)=& \dfrac14\tr((a^Ta)^{-1}(a_t^Ta + a^Ta_t)(a^Ta)^{-1}(a_t^Ta + a^Ta_t))\\
		=& \tr(a_t(a^Ta)^{-1}a_t^T),
	\end{align}
	and
	\begin{align}
		\dfrac12\tr(g^{-1}g_t) &= \dfrac14\tr((a^Ta)^{-1}(a_t^Ta + a^Ta_t)(a^Ta)^{-1}(a^Ta + a^Ta))\\
		=& \tr(a_t(a^Ta)^{-1}a^T) = \tr(a_ta^+).
	\end{align}
	The conclusion follows.
\end{proof}
Using Theorem~\ref{thm.SectionalCurvature} and O'Neill's curvature formula we  obtain the curvature of the space of symmetric matrices, which agrees with the formula of \cite{freed1989basic}:
 \begin{cor}
 The space $\left(\operatorname{Sym}_+(m),\langle \cdot, \cdot \rangle^{\operatorname{Sym}}\right)$ has non-positive sectional curvature given by:
 \begin{align*}
 	\mathcal K_g^{\operatorname{Sym}}(h,k) =& \dfrac{1}{16}\big[\tr([g^{-1}h,g^{-1}k]^2) + \frac{m}{4}\left(\tr(g^{-1}hg^{-1}k)\right)^2 \\
 	&\qquad\qquad-\frac{m}{4}\tr\left((g^{-1}h)^2\right)\tr\left((g^{-1}k)^2\right)\big]\sqrt{\det(g)}
 \end{align*}
 \end{cor}
\begin{proof}
	Similarly as in Section~\ref{sec:symmetric}, we identify the space of symmetric matrices $\operatorname{Sym_+(m)}$ with the quotient space $\operatorname{SO}(n)\backslash M_+(n,m)$.
	Using the fact that the metrics on $M_+(m,n)$ and $\operatorname{Sym}(m)$ are connected via a Riemannian submersion, we can calculate the curvature of the quotient space using O'Neill's curvature formula.
	
	Let $g\in\operatorname{Sym}_+(m)$ and $h,k\in T_{g}\operatorname{Sym}_+(m)$ be two orthonormal tangent vectors with respect to the metric \eqref{eq.metric.sym}. Then we have a lift $a\in M_+(n,m)$ and the horizontal lifts $\tilde{h},\tilde{k}\in T_a\left(M_+(n,m)\right)$  of $h,k$ such that
	\begin{align*}
		\pi(a) = g,\quad d\pi_a(\tilde{h}) = h,\quad d\pi_a(\tilde{k}) = k.
	\end{align*}
	Since $d\pi_a$ is an isometry, $\tilde{h},\tilde{k}$ are orthonormal with respect to the metric \eqref{eq.metric.a}. Recall from Theorem~\ref{thm:matrices:submersion} that any horizontal tangent vector $u$ at $a\in M_+(n,m)$ has the property that $U = ua^+$ is symmetric.
	Thus by Theorem~\ref{thm.SectionalCurvature} the sectional curvature $\mathcal K$ at $a$ for $\tilde{h},\tilde{k}\in T_a(M_+(n,m))$ is given by:
	\begin{multline}
		\mathcal K_a(\tilde{h},\tilde{k})\\
		= \left(\frac74\tr\left([\tilde{H}_0,\tilde{K}_0]^2\right) + \frac{m}{4}\left(\tr(\tilde{H}_0\tilde{K}_0)\right)^2 - \frac{m}{4}\tr(\tilde{H}_0^2)\tr(\tilde{K}_0^2)\right)\sqrt{\det(a^Ta)}.
	\end{multline}
	It remains to calculate O'Neill's curvature term. We have
	\begin{align}
		[\tilde{h},\tilde{k}]a^+ = (\tilde{h}a^+\tilde{k} - \tilde{k}a^+\tilde{h})a^+ = \tilde{H}\tilde{K} - \tilde{K}\tilde{H} = [\tilde{H},\tilde{K}],
	\end{align}
	where the commutator on the right side is the usual matrix commutator, which is defined for any two square matrices. Notice that for symmetric $\tilde{H}$ and $\tilde{K}$, the commutator $[\tilde{H}, \tilde{K}]$ is skew-symmetric and thus $[\tilde{h}, \tilde{k}] = \tilde{h}a^+\tilde{k} - \tilde{k}a^+\tilde{h}$ is in the vertical bundle. Therefore the O'Neill term is given by
	\begin{align}
		\frac34\langle [\tilde{H}, \tilde{K}],[\tilde{H},\tilde{K}]\rangle_a = -\frac34\tr([\tilde{H}, \tilde{K}]^2)\sqrt{\det(a^Ta)}.
	\end{align}
    Notice that $\tr([\tilde{H}, \tilde{K}]^2) = \tr([\tilde{H}_0, \tilde{K}_0]^2)$. Using O'Neill's curvature formula we then obtain the sectional curvature on the quotient space:
    \begin{align*}
    	&\mathcal K_g^{\operatorname{Sym}}(h,k)\\
    =& \left(\tr\left([\tilde{H}_0,\tilde{K}_0]^2\right) + \frac{m}{4}\left(\tr(\tilde{H}_0\tilde{K}_0)\right)^2 - \frac{m}{4}\tr(\tilde{H}_0^2)\tr(\tilde{K}_0^2)\right)\sqrt{\det(a^Ta)}.
    \end{align*}
    It is straightforward calculation to show that
    \begin{align*}
    	\tr\left([\tilde{H}_0,\tilde{K}_0]^2\right) =& \dfrac{1}{16}\tr([g^{-1}h,g^{-1}k]^2);\\
    	\tr(\tilde{H}_0\tilde{K}_0) =& \dfrac{1}{4} \tr(g^{-1}hg^{-1}k);\\
    	\tr(\tilde{H}_0^2)\tr(\tilde{K}_0^2) =& \dfrac{1}{16}\tr\left((g^{-1}h)^2\right)\tr\left((g^{-1}k)^2\right).
    \end{align*}
    Therefore, the result follows.
\end{proof}

\bibliographystyle{abbrv}
\bibliography{refs}
\end{document}